\RequirePackage{snapshot} 
\documentclass[graybox]{svmult}

\usepackage[english]{babel} 
\usepackage{mathptmx}       
\usepackage{helvet}         
\usepackage{courier}        
\usepackage{type1cm}        
\usepackage{amssymb,amsmath}         
\usepackage{makeidx}         
\usepackage{graphicx}        
\usepackage{multicol}        
\usepackage[bottom]{footmisc}

\usepackage{tikz}
\usetikzlibrary{shapes,arrows}
\usetikzlibrary{positioning,fit}

\usepackage[pdftex,colorlinks,%
           bookmarks=true,bookmarksopen=true,%
           pdftitle=Clustering\ comparison\ of\ point\ processes\ with applications\ to\ random\ geometric\ models,%
            pdfauthor=B.\ Blaszczyszyn%
          \ -\  D.\ Yogeshwaran]{hyperref}
\hypersetup{
   bookmarksnumbered,
   pdfstartview={FitH},
   citecolor={blue},
   linkcolor={red},
   urlcolor=[rgb]{0,0.55,0},
   pdfpagemode={UseOutlines}
}

\usepackage{subfigure}

\DeclareMathAlphabet{\mathcal}{OMS}{cmsy}{m}{n}

\newcommand{\C}{\ensuremath{\mathbb{C}}} 
\newcommand{\N}{\ensuremath{\mathbb{N}}} 
\newcommand{\R}{\ensuremath{\mathbb{R}}} 
\newcommand{\Z}{\ensuremath{\mathbb{Z}}} 

\renewcommand{\phi}{\varphi}

\renewcommand{\E}{\ensuremath{\mathbf{E}}} 
\newcommand{\cov}{\ensuremath{\mathbf{c}\mathbf{o}\mathbf{v}}} 
\newcommand{\var}{\ensuremath{\mathbf{v}\mathbf{a}\mathbf{r}}} 
\renewcommand{\P}{\ensuremath{\mathbf{P}}} 
\newcommand{\ind}{\ensuremath{\mathbf{1}}} 

\newcommand{\Pois}{{\rm Pois}} 

\newcommand{\Binom}{{\rm Binom}}

\newcommand{\Tb}{\ensuremath{\mathbb{T}}}

\newcommand{\Ep}{\ensuremath{\mathbb{E}}}

\addto\captionsenglish{%
}

\newcommand{\Ks}{\ensuremath{\mathsf{K}}}

\newcommand{\pr}[1]{\P\{ {#1}\}}

\newcommand{\kommentar}[1]{}

\newcommand\be{\begin{equation}}
\newcommand\ee{\end{equation}}
\def\qed{\hfill\hbox{${\vcenter{\vbox{
    \hrule height 0.4pt\hbox{\vrule width 0.4pt height 6pt
    \kern5pt\vrule width 0.4pt}\hrule height 0.4pt}}}$}}

\makeindex             

\newcommand{\remove}[1]{}
\newcommand{\ur}{{\underline r}}
\newcommand{\ovr}{{\overline r}}

 \def\mE{\mathbb{E}}
\newcommand{\Hex}{\mathbb{H}}
\newcommand{\mL}{\mathbb{L}}
\newcommand{\md}{\mathrm{d}}
\newcommand{\cX}{\mathcal{X}}
\newcommand{\cL}{\mathcal{L}}
\newcommand{\cN}{\mathcal{N}}
\renewcommand{\pr}[1]{\P\{\,#1\,\}}
\newcommand{\EXP}[1]{\E\left(#1\right)}
\newcommand{\COV}[1]{\cov\left(#1\right)}

\newcommand{\al}{\alpha}
\newcommand{\Cox}{Cox}
\newcommand{\detLambda}{\Lambda}

\newcommand{\rndLambda}{\mathfrak{L}}
\newcommand{\rndlambda}{\xi}
\newcommand{\rndB}{\mathfrak{B}}
\newcommand{\cV}{\mathcal{V}}

\def\rar{\rightarrow}
\newcommand{\resp}[1]{\,(\text{resp.}\,#1)}
\newcommand{\NBinom}{\mathrm{NBionom}}
\newcommand{\Geo}{\mathrm{Geo}}
\newcommand{\HGeo}{\mathrm{HGeo}}

\spnewtheorem{condition}{Condition}{\bfseries}{\itshape}
\spnewtheorem*{proofsketch}{Sketch of Proof}{\itshape}{\rmfamily}
\newcommand{\mC}{\mathcal{C}}
\newcommand{\De}{\Delta}
\newcommand{\X}{\mathcal{X}}
\newcommand{\lam}{\lambda}

\begin{document}

\title*{Clustering comparison of point processes with applications to random geometric models}
\titlerunning{Clustering comparison of point processes} 

\author{Bart{\l}omiej B{\l}aszczyszyn and Dhandapani Yogeshwaran}
\institute{B. B{\l}aszczyszyn \at Inria/ENS, 23 av. d'Italie, 75214 Paris France \email{Bartek.Blaszczyszyn@ens.fr}
\and D. Yogeshwaran \at Dept. of Electrical Engineering, Technion -- Israel Institute of Technology, Haifa
32000, ISRAEL
 \email{yogesh@ee.technion.ac.il}}
%
%
\maketitle

\abstract{In this chapter we review some examples, methods, and recent
  results involving comparison of clustering properties of point processes.
Our approach is founded on some basic observations allowing us to consider 
void probabilities and moment measures as two complementary tools for 
capturing clustering phenomena in point processes. As might be expected,
smaller values of these characteristics indicate less
clustering. Also, various global and local functionals of random
geometric models driven by point processes admit more or less explicit
bounds involving void probabilities and moment measures, thus
aiding the study of impact of clustering of the underlying point process.
When stronger tools are needed, directional convex ordering of
point processes happens to be an appropriate choice, as well as 
the notion of (positive or negative) association, 
when comparison to the Poisson point
process is considered. We explain the relations between these tools and
provide examples of point processes admitting them.
Furthermore, we sketch some recent results obtained using the aforementioned
comparison tools, regarding 
percolation and coverage properties of the Boolean model, the SINR model,
subgraph counts in random geometric graphs, and more generally, U-statistics of point processes.
We also mention some results on Betti numbers for \v{C}ech and Vietoris-Rips random complexes generated by stationary point processes.
A general observation is that many of the results derived previously for
the Poisson point process generalise to some ``sub-Poisson'' 
processes, defined as those clustering less than the Poisson process in the sense of 
void probabilities and moment measures, negative association or dcx-ordering.
}

\label{chapter.blaszczyszyn}

\section{Introduction}
\label{blaszcyszyn_sec:introduction}

On the one hand, various interesting methods have been developed for
studying local and global functionals of geometric structures driven by Poisson or Bernoulli point
processes (see \cite{MeeRoy96,Penrose03,Yukich12}). On the other
hand, as will be shown in the following section, there are 
many examples of interesting point processes that occur
naturally in theory and applications. So, the obvious question arises how
much of the theory developed for Poisson or Bernoulli point processes
can be carried over to other classes of point processes. 

Our approach  to this question is based on the comparison of 
 clustering  properties of  point processes. Roughly speaking, a set of points in $\R^d$
clusters if it lacks spatial homogeneity, i.e.,
one observes points forming groups which are
well spaced out. Many interesting properties of random geometric
models driven by point processes should depend
on the ``degree'' of clustering.  For example, it is natural to expect  
that concentrating  points of a point process in well-spaced-out
clusters should negatively impact connectivity of the corresponding
random geometric (Gilbert) graph, and that spreading these clustered points ``more homogeneously'' in the space would
result in a smaller critical radius for which the graph percolates.
For many other functionals, using similar heuristic arguments
one can conjecture whether increase or decrease of clustering will increase or decrease the
value of the functional. However, to the best of our knowledge, there has been no systematic approach towards making these arguments rigorous.

The above observations suggest the following program.
We aim at identifying a class or classes of point
processes, which can be compared in the sense of clustering
to a (say homogeneous) Poisson point 
process, and for which ---  by this comparison --- some results known for
the latter process can be extrapolated. 
In particular, there are point processes which in some sense cluster
less  (i.e. spread their points more homogeneously in the space) than
the Poisson point process. We call them 
{\em sub-Poisson}\index{point process!sub-Poisson}. 
Furthermore, we hasten to explain that the usual strong stochastic order  (i.e. coupling as a subset of the Poisson process) is in general not an appropriate tool in this context.

Various approaches to mathematical formalisation of clustering will
form an important part of this chapter. By formalisation, we mean defining a partial order on the space of point processes such that being smaller with respect to the order indicates less clustering. The most simple approach consists  in considering void probabilities and moment measures
as two complementary tools for capturing clustering phenomena in point
processes. As might be expected, smaller values of these
characteristics indicate less clustering. 
When stronger tools are needed,  directionally convex (dcx) ordering of
point processes happens to be a good choice, as well as 
the notion of negative and positive association. 
Working with these tools, we first give  some  useful, generic 
inequalities regarding  Laplace transforms 
of the compared point processes. In the case of dcx-ordering these
inequalities can be generalised to dcx functions of shot-noise 
fields.

Having described  the clustering comparison tools, 
we present several particular results obtained by using them.
Then, in more detail, we study percolation in the Boolean
model (seen as a random geometric graph) and the SINR graph. In particular, we show
how the classical results regarding the existence of a non-trivial phase
transition extend to models based on additive functionals of
various sub-Poisson point processes. Furthermore, we briefly discuss some
applications of the comparison tools to U-statistics of point processes, 
counts of sub-graphs
and simplices in random geometric graphs and simplicial complexes,
respectively. We also mention results on Betti numbers
of  \v{C}ech and Vietoris-Rips random complexes generated by sub-Poisson
point processes.

Let us conclude the above short motivating 
introduction by citing 
an excerpt from a standard reference on stochastic comparison methods
by M\"uller and Stoyan~(2002):
``
It is clear that there are processes of comparable
variability. Examples of such processes are a homogeneous Poisson
process and a cluster process of equal intensity or two hard-core
Gibbs processes of equal intensity and  different hard-core
distances. It would be fine if these variability differences could be
characterized by order relations ... [implying]%
, for  example,
reasonable relationship[s] for second order characteristics such as the
pair correlation function.''; cf~\cite[page~253]{Muller02}. 
We believe that the results reported in this chapter present one of the first
answers to the above  call, although ``Still much work has to be done in the
comparison theory for point processes.'' (ibid.)

The present chapter is organised as follows.
In Sect.~\ref{blaszcyszyn_sec:examples} we recall some examples of point
processes. The idea is to provide as many as possible examples which
admit various 
clustering comparison orders, described then in Sect.~\ref{blaszcyszyn_sec:methods}.
An extensive overview of applications is presented in Sect.~\ref{blaszcyszyn_sec:results}.

\section{Examples of Point Processes}
\label{blaszcyszyn_sec:examples}

In this section we give some examples of point processes, where
our goal is to present them in the context of modelling of clustering phenomena.
Note that throughout this chapter, we consider point processes on the $d$-dimensional
Euclidean space $\R^{d}$, $d\ge1$, although  much of the presented
results have straightforward extensions to point processes on an arbitrary Polish space.  

\subsection{Elementary Models}
 A first example we probably think
of when trying to come up with some spatially homogeneous model is a point process on a deterministic lattice. 
\begin{definition}[Lattice point process]
\label{blaszcyszyn_def.lattice}
\index{point process!lattice|seealso{lattice process}} 
\index{point process!lattice}
\index{lattice process}
By a lattice point process $\Phi_\mL$ we mean a  simple point process
whose points are located on the vertices of some deterministic
lattice $\mL$.
An important special case is the $d$-dimensional {\em cubic lattice} 
$\Delta\Z^d=\{\Delta(u_1,\ldots,u_d), u_i\in\Z\}$ of edge length $\Delta>0$, where $\Z$ denotes the
set of integers. Another specific (two-dimensional) model 
is the {\em hexagonal lattice} on the complex plane given
by  $\Delta\Hex=\{\Delta(u_1+u_2e^{\emph{i}\pi/3}),\
u_1, u_2 \in\Z\}$.
The stationary version $\Phi_\mL^{\text{st}}$of a lattice point process
$\Phi_\mL$ can be constructed by randomly shifting the deterministic
pattern through a vector $U$
uniformly distributed in some fixed cell of the lattice $\mL$, i.e.
$\Phi_\mL^{\text{st}}=\Phi_\mL + U.$ 
Note that the intensity of the stationary lattice point process is equal to the
inverse of the cell volume. In particular, the
intensity $\lambda_{\Delta\Z^d}$ of  $\Phi_{\Delta\Z^d}^{\text{st}}$ is equal to
$\lambda_{\Delta\Z^d}=1/\Delta^d$, while that of
$\Phi_{\Delta\Hex}^{\text{st}}$ is equal to $\lambda_{\Delta\Hex}=2/(\sqrt
3/\Delta^2)$. 
\end{definition}

Lattice point processes are usually considered to model  ``perfect'' or ``ideal''
structures, e.g. the hexagonal lattice on the complex plane is used to study perfect
cellular communication networks. We will see however, without
further modifications, they escape 
from the clustering comparison methods presented in 
Sect.~\ref{blaszcyszyn_sec:methods}.

When the ``perfect structure'' assumption cannot be retained and one
needs a random pattern, then the Poisson point process usually comes as a natural 
first modelling  assumption.  
We therefore recall the definition of the Poisson process for the convenience of the reader, see also the survey given in \cite{baddeley12}.

\begin{definition}[Poisson point process]
\index{point process!Poisson|seealso{Poisson process}}
\index{point process!Poisson}
\index{Poisson process}
Let $\detLambda$ be a (deterministic) locally finite measure
on the Borel sets of $\R^d$. 
The random counting measure $\Pi_{\Lambda}$ is called a {\em  Poisson point process} with intensity measure
$\Lambda$ if for every $k=1,2,\ldots$ and 
all bounded, mutually  disjoint Borel sets
$B_1,\allowbreak\ldots,B_k$, the random variables 
$\Pi_{\Lambda}(B_1),\allowbreak\ldots,\Pi_{\Lambda}(B_k)$ are 
independent, with Poisson distribution
$\Pois(\Lambda(B_1)),\ldots, \Pois(\Lambda(B_k))$, respectively.
In the case when $\Lambda$ has an integral representation
$\Lambda(B)=\int_B\lambda(x)\,\md x$, 
where $\lambda: \R^d \rightarrow \R_+$ is some measurable function, 
we call $\lambda$
the \textit{intensity field} of the Poisson point process.
In particular, if $\lambda$ is a constant,
we call $\Pi_{\Lambda}$ a 
{\em homogeneous} Poisson point process and denote it by  $\Pi_\lambda$.
\end{definition}

The Poisson point process is a good model when one does not expect
any ``interactions'' between points. This is related to the
{\em complete randomness} property of  Poisson processes, cf.~\cite[Theorem 2.2.III]{DVJI2003}.
\begin{svgraybox}
The homogeneous Poisson point process is commonly considered as a reference model
in comparative studies of   clustering phenomena.
\end{svgraybox}

\subsection{Cluster Point Processes --- Replicating and Displacing Points}
\label{blaszcyszyn_ss.perturbation}
We now present several operations on the points of a
point process,
which in conjunction with the two elementary models presented above
allow us to construct various other interesting examples of point processes.
We begin by recalling the  following elementary operations.

Superposition of  patterns of points consists of set-theoretic
addition of these  points. Superposition of (two or more) point processes
$\Phi_1,\ldots,\Phi_n$, defined as random counting measures, consists of adding
these measures $\Phi_1+\cdots+\Phi_n$. 
Superposition of independent  point processes is of special interest. 

Thinning of a point process  consists of suppressing some subset of its
points. {\em Independent thinning} with retention function $p(x)$
defined on $\R^d$, $0\le p(x)\le 1$, 
consists in suppressing points independently, given a realisation of the 
point process, with probability $1-p(x)$, which might depend on the
location $x$  of the point to be suppressed or not. 

Displacement consists in, possibly randomised, mapping 
of points of a  point process to the same or another space. {\em Independent
  displacement} with displacement (probability) kernel $\cX(x,\cdot)$
from $\R^d$ to $\R^{d'}$, $d'>1$,  consists of  
independently mapping each point $x$ of a given realisation of the
original point process  to a new, random location in 
$\R^{d'}$ selected according to the kernel $\cX(x,\cdot)$. 

\begin{remark}
An interesting property of the class of Poisson point processes
is that it is closed with respect to independent displacement, thinning
and superposition, i.e. the result of these operations made on  Poisson
point processes is  a Poisson point process, which is not granted in the
case of an arbitrary (i.e. not independent) superposition or displacement,
see e.g. \cite{DVJI2003, DVJII2007}. 
\end{remark}

Now, we define a more sophisticated operation on  point processes that will
allow us to construct new classes of point processes
with interesting clustering properties.
\begin{definition}[Clustering perturbation of a point process]
\index{point process!clustering perturbation}
Let $\Phi$ be a point process on $\R^d$ and $\cN(\cdot,\cdot)$,
$\cX(\cdot,\cdot)$ be two probability kernels from $\R^d$ to
the set of non-negative integers $\Z_+$ and $\R^{d'}$, $d,d'\ge1$, respectively.
Consider the following subset of $\R^{d'}$. Let
\begin{equation}
\label{blaszcyszyn_defn:perturbed_pp}
\Phi^{\text{pert}}= \bigcup_{X \in \Phi} \bigcup_{i=1}^{N_X} \{Y_{iX}\}\,,
\end{equation}
where, given $\Phi$, 
\begin{itemize}
\item[1.] $(N_X)_{X\in\Phi}$ are independent, non-negative
integer-valued random variables with (conditional) distribution
$\P \left( N_X\in\cdot\,|\,\Phi \right)=\cN(X,\cdot)$, 
\item[2.] ${Y}_{X}=(Y_{iX}; i =1,2,\ldots)$, $X\in\Phi$ are
  independent vectors of  i.i.d.  random elements of  $\R^{d'}$, with
  $Y_{iX}$ having the conditional distribution
  $\P \left( Y_{iX}\in\cdot\,|\,\Phi \right)=\cX(X,\cdot)$. 
\end{itemize}

Note that the inner sum in~(\ref{blaszcyszyn_defn:perturbed_pp}) is interpreted as
$\emptyset$ when $N_X=0$.
\end{definition}

The random set  $\Phi^{\text{pert}}$ given in (\ref{blaszcyszyn_defn:perturbed_pp}) can be considered as a point process
on $\R^{d'}$ provided it is a locally finite. In what follows, 
we will assume a stronger condition, namely that the itensity measure of 
$\Phi^{\text{pert}}$ is locally finite (Radon), i.e.
\begin{equation}
\label{blaszcyszyn_e:perturbation-Radon}
\int_{\R^d} n(x)\cX(x,B)\,  \al(\md x)< \infty,
\end{equation}
for all bounded Borel sets $B\subset\R^{d'}$, where  $\al(\cdot)=\EXP \Phi(\cdot) $ denotes the intensity measure of  $\Phi$ and
\begin{equation}\label{bl_equation_1}
n(x)=\sum_{k=1}^\infty 
k\cN(x,\{k\})
\end{equation}
is the mean value of the distribution $\cN(x,\cdot)$.

\begin{svgraybox}
{\em Clustering perturbation} of a given parent process $\Phi$ consists in 
independent replication and
displacement of the points of $\Phi$, with the  number of
replications of a given point  $X\in\Phi$ having distribution
$\cN(X,\cdot)$ and the replicas' locations having distribution
$\cX(X,\cdot)$. The replicas of $X$ form a {\em cluster}.
\end{svgraybox}

For obvious reasons, we call $\Phi^{\text{pert}}$ a {\em perturbation} of~$\Phi$
driven by  the {\em replication kernel} \index{perturbation kernel!replication}
$\cN$ and the {\em displacement
  kernel} \index{perturbation kernel!displacement} $\cX$. 
It is easy to see that the independent thinning and 
displacement operations described above are special cases of
clustering perturbation. In what follows we present a few known examples of point processes arising as  clustering 
perturbations of a lattice or Poisson
point process. For simplicity we assume that replicas stay in the same
state space, i.e.  $d=d'$.

\begin{example}[Binomial point process]
\label{blaszcyszyn_ex.Binomial}
\index{point process!binomial|seealso{binomial process}}
\index{point process!binomial}
\index{binomial process}
A (finite) binomial point process has a fixed total number  $n<\infty$
of points, which are independent and identically distributed according to
some (probability) measure $\Lambda$ on $\R^d$.
It can be seen as a Poisson  point process $\Pi_{n\Lambda}$  
conditioned to have $n$ points, cf. \cite{DVJI2003}. 
Note that this property might be seen as a clustering perturbation of a
one-point process, with deterministic  number $n$ of point replicas
and displacement distribution $\Lambda$.
\end{example}  

\begin{example}[Bernoulli lattice]
The Bernoulli lattice arises as independent thinning of a lattice point
process; i.e.,
each point of the lattice is retained (but not displaced)
with some probability $p \in (0,1)$ and suppressed otherwise.
\end{example}

\begin{example}[Voronoi-perturbed lattices]
\label{blaszcyszyn_ex:Voronoi_pert_lattice}
\index{lattice process!Voronoi-perturbed}
These are  perturbed lattices  with displacement kernel $\cX$,
where the distribution
$\cX(x,\cdot)$ is supported on the Voronoi cell $\cV(x)$ of vertex $x\in \mL$
of the original (unperturbed) lattice $\mL$.
In other words,   each replica  of a given lattice point  gets
independently translated  to some random location chosen
in the Voronoi cell  of the original lattice point.
Note that one can also choose other bijective, lattice-translation invariant
mappings of associating lattice cells to
lattice  points; e.g. associate a given cell of the square lattice  
on the plane $\R^2$ to its ``south-west''  corner.

By a {\em simple perturbed lattice} 
\index{lattice process!simple perturbed} we mean the Voronoi-perturbed
lattice whose points are uniformly translated  in the corresponding
cells, without being replicated. 
Interestingly enough, 
the Poisson point process $\Pi_\Lambda$ with some intensity measure $\Lambda$ 
can be constructed 
as a Voronoi-perturbed lattice. 
\index{Poisson process!as perturbed lattice} 
\index{perturbation kernel!replication!Poisson}
Indeed, it is enough to take
the Poisson replication kernel $\cN$ given by $\cN(x,\cdot)=\Pois(\Lambda(\cV(x)))$ and 
the displacement kernel $\cX$ with $\cX(x,\cdot)=\Lambda(\cdot\cap
\cV(x))/\Lambda(\cV(x))$; cf.~Exercise~\ref{ex.lattice-to-Poisson}.
Keeping the above  displacement kernel and 
replacing the Poisson distribution in the replication  kernel
by some other distributions
convexly smaller or larger than the  Poisson distribution, one gets the following
two particular classes of Voronoi-perturbed lattices, clustering their points
less or more than the Poisson point process $\Pi_\Lambda$ (in a sense
that will be formalised in Sect.~\ref{blaszcyszyn_sec:methods}).

{\em Sub-Poisson Voronoi-perturbed lattices}
\index{point process!sub-Poisson}
\index{lattice process!sub-Poisson perturbed} 
\index{perturbation kernel!convex ordering}
are 
Voronoi-perturbed lattices such that   $\cN(x,\cdot)$ is convexly
smaller  than $\Pois(\Lambda(\cV(x)))$.
Examples of distributions convexly smaller than $\Pois(\lambda)$ are the 
  hyper-geometric distributions \index{perturbation kernel!replication!hyper-geometric}
 $\HGeo(n,m,k)$,
$m,k\le n$, $km/n=\lambda$
and the binomial distributions. \index{perturbation kernel!replication!binomial}
 $\Binom(n,\lambda/n)$, $\lambda\le n$%
,
which can be ordered as follows:
\begin{equation}
\label{blaszcyszyn_eqn:sub_Poisson_rvs}
\HGeo(n,m,\lambda n/m)\le_{\text{cx}} \Binom(m,\lambda/m)\le_{\text{cx}}
\Binom(r,\lambda/r)\le_{\text{cx}} \Pois(\lambda),
\end{equation}
for $\lambda\le m\le \min(r,n)$;
cf.~\cite{Whitt1985}. Recall that
$\Binom(n,p)$ has the probability mass function
  $p_{\Binom(n,p)}(i)={n\choose i}p^i(1-p)^{n-i}$ ($i=0,\ldots,n$), whereas 
$\HGeo(n,m,k)$ has the probability
  mass function $p_{\HGeo(n,m,k)}(i)={m\choose i}{n-m\choose
    k-i}/{n\choose k}$ ($\max(k-n+m,0)\le i\le m$);
cf. Exercise~\ref{ex.convex-order}. 

{\em Super-Poisson Voronoi-perturbed lattices}
\index{lattice process!super-Poisson perturbed}
are \label{blaszcyszyn_sec:sup_poisson_lattice} Voronoi-perturbed lattices 
with  $\cN(x,\cdot)$ convexly larger than $\Pois(\Lambda(\cV(x)))$.
Examples of distributions convexly larger than $\Pois(\lambda)$ are the
negative binomial distribution ${\NBinom}(r,p)$ 
\index{perturbation kernel!replication!negative binomial}
with $rp/(1-p)=\lambda$ and the geometric distribution $\Geo(p)$ 
\index{perturbation kernel!replication!geometric}
distribution
$1/p-1=\lambda$, 
which can be ordered in the following way:
\begin{eqnarray}
\label{blaszcyszyn_eqn:super_Poisson_rvs}
\Pois(\lambda) & \le_{\text{cx}} & \NBinom(r_2,\lambda/(r_2+\lambda))
\le_{\text{cx}} \NBinom(r_1,\lambda/(r_1+\lambda)) \nonumber \\
\label{blaszcyszyn_eqn:sup_Poisson_rvs} & \le_{\text{cx}} & \Geo(1/(1+\lambda)) 
\le_{\text{cx}}\sum_{j}\lambda_j\;\Geo(p_j) 
\end{eqnarray}
with $r_1\le r_2$,
$0\le \lambda_j\le 1$, $\sum_j\lambda_j=1$ and
$\sum_j\lambda_j/p_j=\lambda+1$,
where the largest distribution in (\ref{blaszcyszyn_eqn:super_Poisson_rvs}) is a mixture
of geometric distributions having mean~$\lambda$. Note that any mixture of Poisson distributions 
\index{perturbation kernel!replication!Poisson mixture}
having mean $\lambda$ is in cx-order larger than $\Pois(\lambda)$.
Furthermore, recall that the probability mass functions of $\Geo(p)$ and $\NBinom(r,p)$ are given
by
$p_{\Geo(p)}(i)=p(1-p)^{i}$ and  $p_{\NBinom(r,p)}(i)=
{r+i-1\choose i}p^i(1-p)^r$, respectively.
\end{example}

\begin{example}[Generalised shot-noise Cox point processes]
\label{blaszcyszyn_ex.GSNCox}
\index{Cox process!generalised}
These are clustering perturbations of an arbitrary parent point
process $\Phi$,  with replication kernel $\cN$, where $\cN(x,\cdot)$
is the Poisson distribution $\Pois(n(x))$
and $n(x)$ is the mean value given in (\ref{bl_equation_1}).
Note that in this case, given $\Phi$,  
the clusters (i.e. replicas of the given parent
point) form independent Poisson point process
$\Pi_{n(X)\cX(X,\cdot)}$, $X\in\Phi$.
This special class of  Cox point processes (cf.
Sect.~\ref{blaszcyszyn_ss.Cox}) has been  introduced in~\cite{Moller05}.
\end{example}

\begin{example}[Poisson-Poisson cluster point processes]
\label{blaszcyszyn_ex.PP-cluster}
This is a special case of the generalised shot-noise Cox point processes,
with the parent point process being  Poisson, i.e. $\Phi = \Pi_\Lambda$
for some intensity measure $\Lambda$.
A further special case is often discussed in the literature, where the displacement kernel $\cX$ 
is such that $\cX(X,\cdot)$ is the uniform distribution
in  the ball $B_r(X)$ of some given radius $r$. It is called the
{\em  Mat\'ern cluster \index{Cox process!Matern@Mat\'ern cluster}
 point process}. 
If  $\cX(X,\cdot)$ is symmetric Gaussian, then
the resulting Poisson-Poisson cluster point process is called a {\em (modified)
  Thomas \index{Cox process!Thomas (modified)}
 point process}.  
\end{example}

\begin{example}[Neyman-Scott point process]
\label{blaszcyszyn_ex.NS-PP}
\index{point process!Neyman-Scott|seealso{Neyman-Scott process}}
\index{point process!Neyman-Scott}
\index{Neyman-Scott process}
These point processes arise as  
a clustering  perturbation of a Poisson parent point process
$\Pi_\Lambda$, with arbitrary (not necessarily Poisson)
replication kernel $\cN$. 
\end{example}

\subsection{Cox Point Processes--- Mixing Poisson Distributions}
\label{blaszcyszyn_ss.Cox}
We now consider a rich class of point processes 
known also as doubly stochastic Poisson point process, which 
are  often used to model patterns exhibiting more clustering than 
the Poisson point process. 
\begin{definition}[Cox point process]
\label{blaszcyszyn_d.Cox}
\index{point process!Cox|seealso{Cox process}} 
\index{point process!Cox} 
\index{Cox process}
Let $\rndLambda$  be a random locally finite (non-null)
measure on $\R^d$. A {\em Cox point process} $\Cox_{\rndLambda}$ on $\R^d$
generated 
by $\rndLambda$ is defined as point process having the property that
$\Cox_{\rndLambda}$ conditioned on $\rndLambda=\Lambda$
is the Poisson point process $\Pi_\Lambda$. 
Note that $\rndLambda$ is called the {\em random intensity measure} of $\Cox_{\rndLambda}$. 
In case when the random measure
$\rndLambda$ has an integral representation
$\rndLambda(B)=\int_B\rndlambda(x)\,\md x$, with $\{ \rndlambda(x), x \in \R^d\}$ being
a random field, we call this field  
the {\em random intensity} field of the Cox process. In the special case that
$\rndlambda(x)=\xi\lambda(x)$ for all $x\in \R^d$, where $\xi$ is a (non-negative) random
variable and $\lambda$ a (non-negative) deterministic function, the
corresponding Cox point process is called a {\em mixed Poisson}
\index{Poisson process!mixed} 
 point process.
  \end{definition}

\begin{svgraybox}
Cox processes may be
seen as a result of an operation transforming some  random 
measure $\rndLambda$ into a point process  $\Cox_{\rndLambda}$,
being a mixture of Poisson processes.
\end{svgraybox}

In Sect. \ref{blaszcyszyn_ss.perturbation}, we have already seen that clustering perturbation of an arbitrary 
point process  with Poisson replication kernel gives rise to Cox
processes (cf. Example~\ref{blaszcyszyn_ex.GSNCox}), where 
Poisson-Poisson cluster point processes are  special cases 
with  Poisson parent point process.
This latter class of point processes can be naturally extended by replacing 
the Poisson parent process  by 
a L\'{e}vy basis. 
\begin{definition}[L\'evy basis]
\index{L\'evy basis}
A collection of real-valued random variables  $\{Z(B), B \in \mathcal{B}_0\}$,  
where $\mathcal{B}_0$ denotes the family of bounded Borel sets
in $\R^d$, is said to be a  {\em L\'{e}vy basis} if the $Z(B)$ are  infinitely divisible random variables and 
for any sequence $\{B_n\}$, $n\ge 1$, of disjoint bounded Borel sets in $\R^d$,
$Z(B_1),Z(B_2),\ldots$ are independent random variables ({\em complete
  independence property}), with $Z(\bigcup_n B_n)=\sum_n Z(B_n)$
almost surely provided that $\bigcup_n B_n$ is bounded.
\end{definition}

In this chapter, we shall consider only non-negative L\'{e}vy bases.
We immediately see that the Poisson point process is a special case of a L\'evy
basis. Many other concrete examples of L\'evy bases can be
obtained by ``attaching'' independent, infinitely divisible
 random variables $\xi_i$  to a deterministic, locally finite sequence
 $\{x_i\}$ of (fixed) points in $\R^d$ and letting $Z(B)
 =\sum_{i}\xi_i\ind(x_i\in B)$. In particular, clustering
 perturbations of a lattice, with infinitely divisible replication
 kernel and no displacement (i.e. $\cX(x,\cdot)=\delta_x$, where
 $\delta_x$ is the Dirac measure at $x$) are L\'evy bases.
Recall that any degenerate (deterministic),
Poisson, negative binomial, gamma as well as Gaussian, Cauchy,
Student's distribution are examples of infinitely divisible distributions.

It is possible to define an integral of a measurable function 
with respect to  a L\'evy basis (even if the latter is not always a
random measure; see~\cite{Hellmund08} for details) and consequently
consider the following classes of Cox point processes.  
\begin{example}[L\'{e}vy-based Cox point process]
Consider a  Cox point process
 $\Cox_\rndLambda$ with random intensity field that 
is an integral shot-noise field of a L\'{e}vy basis, i.e.
$\rndlambda(y) = \int_{\R^d}k(x,y)\allowbreak Z (\md x)$,
where $Z$  is a L\'{e}vy basis and $k:\R^d\times \R^d \rightarrow \R_+$ is some 
non-negative function almost surely integrable with respect to
$Z\otimes\md y$.
\end{example}

\begin{example}[Log-L\'{e}vy-based Cox point process]
\index{Cox process!log-Levy@log-L\'evy}
These are Cox point processes  with random intensity field  
given by
$\rndlambda(y) = \exp \left( \int_{\R^d} k(x,y) Z (\md x) \right)$, 
where  $Z$ and $k$ satisfy the same conditions as above.
\end{example}

Both L\'evy- and log-L\'evy-based Cox point processes have been introduced in
\cite{Hellmund08}, where one can find many examples of these processes.
We still mention another class of Cox point processes considered in~\cite{Moller98}.
\begin{example}[Log-Gaussian Cox point process]
\index{Cox process!log-Gaussian}
Consider a  Cox point process whose random intensity field is given by
$\rndlambda(y)= \exp 	\left( \eta(y)\right)$ where $\{ \eta (y) \}$ 
is a Gaussian random field.
\end{example}

\subsection{Gibbs and Hard-Core Point Processes}
Gibbs \index{point process!Gibbs} and hard-core point processes are 
two further classes of point processes, 
which should appear in the  context of modelling of clustering phenomena.  

Roughly speaking Gibbs point processes are point processes having a density with respect to the Poisson point process. In other words, we obtain a Gibbs point process, when we ``filter'' Poisson patterns of points, giving more chance to appear for some configurations and less chance (or completely suppressing) some
others. A very simple example is a Poisson point process conditioned to obey some
constraint  regarding  its points in some bounded Borel set (e.g. to have some given
number of points there). Depending on the ``filtering'' condition we may naturally create point processes which cluster more or less than the Poisson point process. 

Hard-core point processes are point process in which the points are separated from each other by some minimal distance, hence in some sense clustering is
``forbidden by definition''.

However,  we will not give precise definitions, nor present particular
examples from these classes of point processes, because, unfortunately, we do not have
yet  interesting enough comparison results for them,
to be presented in the remaining part of this chapter.

\subsection{Determinantal and Permanental Point Process}
We briefly recall two special classes of point processes arising in random matrix theory, combinatorics, and physics.
They are ``known'' to cluster their points, less or
more, respectively, than the Poisson point process.

\begin{definition}[Determinantal point process]
\label{blaszcyszyn_defn.det}
\index{point process!determinantal|seealso{determinantal process}}
\index{point process!determinantal}
\index{determinantal process}
A simple point process on $\R^d$
is said to be a {\em determinantal point process} with a kernel function $k:\R^d\times\R^d\rightarrow \C$ with respect to
a Radon measure $\mu$ on $\R^d$ if the  joint intensities
$\rho^{(\ell)}$ of the factorial moment measures
of the point process with respect to the product measure $\mu^{\otimes \ell}$ satisfy
$\rho^{(\ell)}(x_1,\ldots,x_\ell) = \det \big( k(x_i,x_j) \big)_{1
\leq i,j \leq \ell}$ for
all $\ell$, where $\big(a_{ij}\big)_{1 \leq i,j \leq \ell}$ stands for a
matrix with entries $a_{ij}$ and $\det$ denotes the
determinant of the matrix. 
\end{definition}

\begin{definition}[Permanental point process]
\label{blaszcyszyn_defn.perm}
\index{point process!permanental|seealso{permanental process}}
\index{point process!permanental}
\index{permanental process}
Similar to the notion of a determinantal point process, one says that a simple point process is a {\em permanental point process} with a kernel function $k:\R^d\times\R^d\rightarrow \C$ with respect to a Radon measure $\mu$ on $\R^d$ if the  joint intensities $\rho^{(\ell)}$ of the point process with respect to $\mu^{\otimes \ell}$
satisfy $\rho^{(\ell)}(x_1,\ldots,x_\ell) = \text{per}\big( k(x_i,x_j) \big)_{1
\leq i,j \leq \ell}$ for all $\ell$, where $\text{per}\big( \cdot \big)$
stands for the permanent of a matrix.  
From~\cite[Proposition~35 and Remark~36]{Ben06}, we know that
each permanental point process is a Cox point process.
\end{definition}

Naturally, the kernel function $k$ needs to satisfy some additional
assumptions for the existence of the point processes defined above. We refer
to~\cite[Chap.~4]{Ben09} 
for a general framework which allows  to study 
determinantal and permanental point processes, see also~\cite{Ben06}.
Regarding statistical aspects and simulation methods for  determinantal
point processes, see \cite{Lavancier12}.

Here is an important example of a determinantal point process
recently studied on the theoretical ground (cf. e.g.~\cite{Goldman2010})
and considered in modelling applications (cf.~\cite{Miyoshi2012}).
\begin{example}[Ginibre point process]
\label{blaszcyszyn_ex.Ginibre}
\index{point process!Ginibre}
\index{determinantal process!Ginibre}
This is the determinantal point process on $\R^2$
with kernel function
$k((x_1,x_2),\allowbreak(y_1,y_2))=\exp[(x_1y_1+x_2y_2)+ i(x_2y_1-x_1y_2)]$,
$x_j,y_j\in\R$,  $j=1,2$,
with respect to the measure
$\mu(\md(x_1,x_2))=\pi^{-1}\exp[-x_1^2-x_2^2]\,\md x_1\md x_2$.
\end{example}

{
\begin{exercise}
\label{ex.lattice-to-Poisson}
Let $\Phi$ be a simple point process on~$\R^d$. Consider its cluster
perturbation $\Phi^{\text{pert}}$ defined in~(\ref{blaszcyszyn_defn:perturbed_pp}) with 
the Poisson  replication kernel $\cN(x,\cdot)=\Pois(\Lambda(\cV(x)))$,
where $\cV(x)$ is the Voronoi cell of $x$ in $\Phi$, and 
the displacement kernel $\cX(x,\cdot)=\Lambda(\cdot\cap
\cV(x))/\Lambda(\cV(x))$, for some given deterministic Radon measure
$\Lambda$ on $\R^d$. Show that $\Phi^{\text{pert}}$ is Poisson with intensity
measure $\Lambda$.
\end{exercise}
}

{
\begin{exercise}
\label{ex.convex-order}
Prove~(\ref{blaszcyszyn_eqn:sub_Poisson_rvs}) {and (\ref{blaszcyszyn_eqn:sup_Poisson_rvs})} by showing 
the logarithmic concavity of the ratio of the respective probability mass
functions, which  implies increasing convex order and, consequently,
cx-order provided the distributions have the same  means.
\end{exercise}
}

\section{Clustering Comparison Methods}
\label{blaszcyszyn_sec:methods}

Let us begin with  the following {\em informal} definitions.
\begin{svgraybox}
A set of points is spatially
homogeneous if approximately the same numbers of points occur in any
spherical region of a given volume. A set of points clusters if it lacks
spatial homogeneity; more precisely, if  one observes points 
arranged in groups being well spaced out.
\end{svgraybox}

Looking at Fig.~\ref{blaszcyszyn_f.Lattice1}, it is  intuitively obvious that
(realisations of) some point processes cluster less than others.
However, the mathematical formalisation of such a statement appears not so easy.
In what follows, we present a few possible approaches. We begin with
the usual statistical descriptors of spatial homogeneity, then show
how void probabilities and moment measures come into the
picture, in particular in relation to another notion useful in this
context:  positive and negative association. Finally we deal with
directionally convex ordering of
point processes. 

This kind of organisation roughly corresponds to presenting ordering methods from weaker to stronger ones; cf. Fig.~\ref{blaszcyszyn_f.flowchart1}.
We also show how the different examples presented in
Sect.~\ref{blaszcyszyn_sec:examples} admit these comparison methods,
mentioning the respective results in their strongest versions.
We recapitulate results regarding comparison to the Poisson process in
Fig.~\ref{blaszcyszyn_f.flowchart2}.

\begin{figure*}[!t]
\begin{center}
\begin{minipage}[b]{0.25\linewidth}
\includegraphics[width=1.\linewidth]{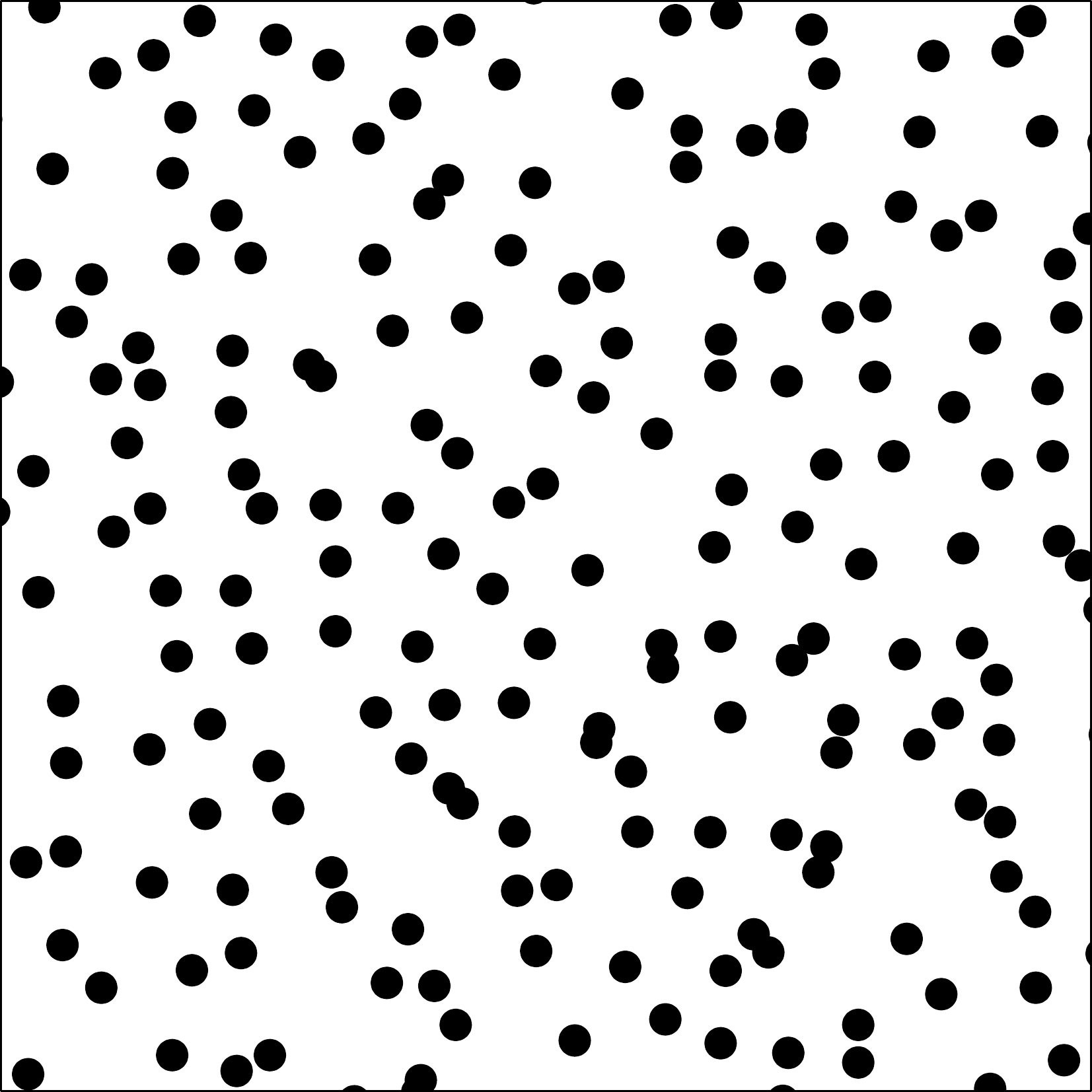}\\
\centerline{simple perturbed lattice}
\end{minipage}
\begin{minipage}[b]{0.25\linewidth}
\includegraphics[width=1.\linewidth]{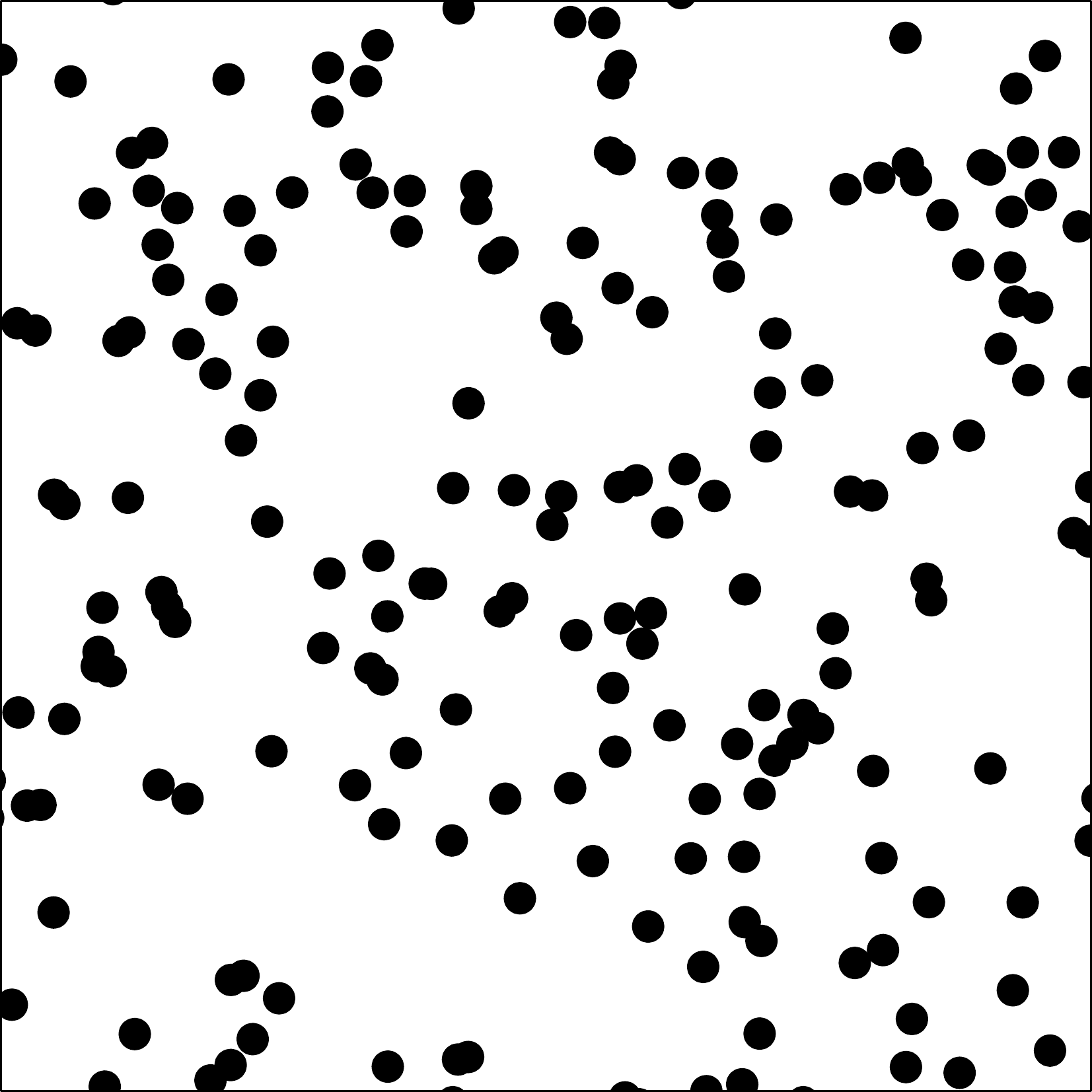}\\
\centerline{Poisson point process}
\end{minipage}
\begin{minipage}[b]{0.25\linewidth}
\includegraphics[width=1.\linewidth]{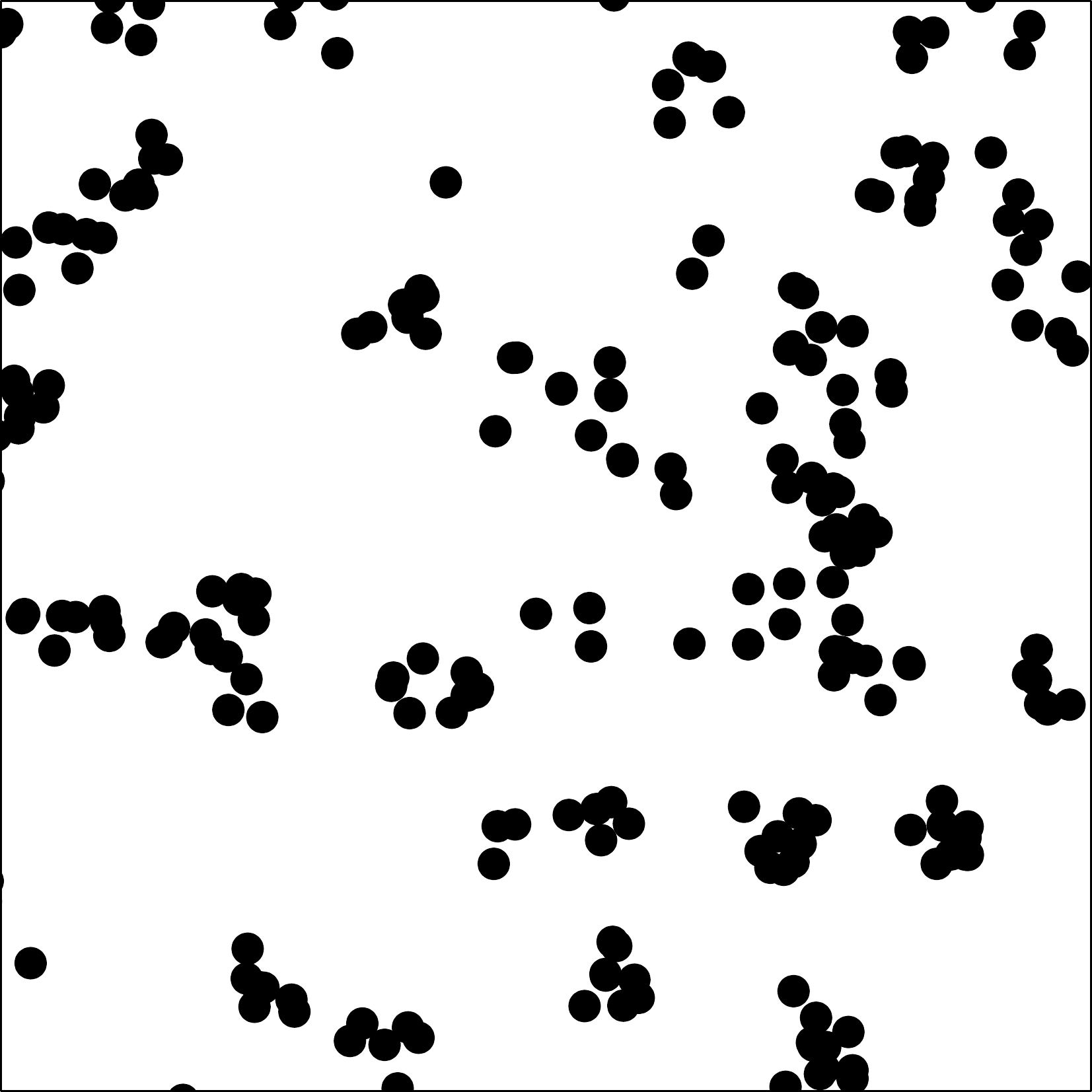}\\
\centerline{Cox point process}
\end{minipage}
\end{center}
\vspace{-2ex}
\caption{\label{blaszcyszyn_f.Lattice1} From left to right : patterns sampled from
a simple  perturbed lattice (cf. Example~\ref{blaszcyszyn_ex:Voronoi_pert_lattice}), Poisson point process and  a doubly stochastic Poisson (Cox) point process, all with the same mean number of points per unit area.} 
  \end{figure*}

\subsection{Second-order statistics}
In this section we restrict ourselves to the stationary setting. 

\subsubsection{Ripley's K-Function}
\label{blaszcyszyn_sss.Ripley}
\index{K function Ripley's@K function (Ripley's)}

One of the most popular functions for the statistical analysis of
spatial homogeneity is {\em Ripley's K-function}
$K:\R_+ \rightarrow \R_+$   defined for stationary point processes (cf.~\cite{stoyetal95}). Assume that $\Phi$ is a stationary point process on $\R^d$ with finite intensity $\lambda=\E{\Phi([0,1]^d)}$. Then
$$K(r)=\frac1{\lambda \nu_d(B)}\EXP \sum_{X_i\in\Phi\cap B}
(\Phi(B_{X_i}(r))-1) \,,$$ where
$\nu_d(B)$ denotes the Lebesgue measure of a bounded Borel set $B\subset \R^d$,
assuming that $\nu_d(B)>0$.
Due to stationarity, the definition does not depend on the choice of
$B$.

The value of $\lambda K(r)$ can be interpreted as the
average number of ``extra'' points  
observed within the distance $r$ from a  randomly chosen (so-called
typical) point.
Campbell's formula from Palm theory of stationary point processes gives a precise
meaning to this statement. Consequently, for a given intensity $\lambda$, the more 
one finds points of a point process located in clusters of
radius $r$, the larger the value of $K(r)$ is, whence a first clustering
comparison method follows. 
\begin{svgraybox}
Larger values $K(r)$ of Ripley's K-function indicate more clustering
``at the cluster-radius scale'' $r$.
\end{svgraybox}
\label{blaszcyszyn_homogeneity}
For the (homogeneous) Poisson  process  $\Pi_\lambda$ on $\R^d$,
 which is often considered as a ``reference  model'' for clustering,
we have $K(r)=\kappa_d r^d$,
where $\kappa_d$ is the volume of the unit ball in $\R^d$.
Note here that   $K(r)$  describes clustering characteristics of the
point process at the (cluster radius) scale $r$. 
Many point processes, which we tend to think that they cluster less or more  than
the Poisson point process, in fact are not comparable in the sense of Ripley's 
K-function (with the given inequality holding for all $r\ge0$, neither, in
consequence, in any stronger sense considered later in this section), as we
can see in the following simple example. 

The following result of D.~Stoyan from 1983
can be considered as a precursor to our theory of clustering
comparison. 
It says that the convex ordering of Ripley's
K-functions implies ordering of variances of number of observed points.
We shall see in Remark \ref{blaszcyszyn_r.moments} that variance bounds give us simple concentration inequalities for
the distribution of the  number of observed points.
These inequalities help to control clustering. We will
develop this idea further in Section~\ref{blaszcyszyn_ss.Moments}
and~\ref{blaszcyszyn_ss.Voids} showing that using moment measures and void
probabilities one can obtain stronger, exponential concentration equalities.
\begin{proposition}[{\cite[Corollary~1]{Stoyan1983inequalities}}]
\label{blaszcyszyn_prop:stoyan1983}
Consider two stationary, isotropic 
point processes $\Phi_1$ and $\Phi_2$ of the same intensity,
with the Ripley's functions $K_1$ and $K_2$, respectively.
If  $K_1\le_{dc} K_2$ i.e., 
$\int_0^\infty f(r)K_1(\md r)\le \int_0^\infty f(r)K_2(\md r)$
for all decreasing convex $f$ then 
$\var(\Phi_1(B))\le \var(\Phi_2(B))$
 for all compact, convex $B$.
\end{proposition}

\remove{
\begin{example}
\index{lattice process!Ripley's K function of@Ripley's $K$ function of} 
For the stationary square lattice point process
on the plane with intensity $\lambda=1$ 
(cf.~Definition~\ref{blaszcyszyn_def.lattice}), presumably more homogeneous (less
clustering)  than
Poisson point process of the same intensity, 
it holds that $K(r)=0\le \pi r^2$ for $r<1$, but $K(r)=4$  
for $r\in[1,\sqrt 2 )$ which is larger than $\pi r^2$ for $r$
approaching~1 from the right. Similarly, $K(r)=8>\pi r^2$  
for $r=\sqrt2$, $K(r)=12>\pi r^2$ for $r=2$, etc.
We thus see that Ripley's $K$-function, is able to
compare point patterns at different ``user-defined'' cluster-radius
scales~$r$, and that for different values of $r$ we may obtain
opposite in-qualities.
However, as we shall see later (cf. Example~\ref{blaszcyszyn_ex.pert_lattice_dcx})
for some perturbed lattices, including the simple perturbed ones,
it holds that $K(r)\le \kappa_d r^d$ for all $r\ge0$ and thus they cluster less than
the Poisson point process in the sense of Ripley's $K$-function
(and even in a much stronger sense). We will also discuss
point processes clustering more than the Poisson processes in this sense.    
\end{example}
}

{
\begin{exercise}\label{blasz.exer}
\index{lattice process!Ripley's K function of@Ripley's K function of} 
For the stationary square lattice point process
on the plane with intensity $\lambda=1$ 
(cf.~Definition~\ref{blaszcyszyn_def.lattice}), compare $K(r)$ and $\pi r^2$ for $r = 1, \sqrt{2},2$. 
\end{exercise}
}

{From Exercise~\ref{blasz.exer}, one should be able to see that though the square lattice is presumably more homogeneous (less clustering) than the Poisson point process of the same intensity, the differences of the values of their K-functions alternate between strictly positive and strictly negative. However, we shall see later that (cf. Example~\ref{blaszcyszyn_ex.pert_lattice_dcx}) this will not be the case for  some perturbed lattices, including the simple perturbed ones and thus they cluster less than
the Poisson point process in the sense of Ripley's K-function
(and even in a much stronger sense). We will also discuss
point processes clustering more than the Poisson processes in this sense. }

\subsubsection{Pair Correlation Function} 
\label{blaszcyszyn_sss.PairCorrelationFunction}
\index{pair correlation function}
Another useful characteristic for measuring clustering effects in
stationary point processes is the {\em pair correlation function}
$g:\R^d\times \R^d \rightarrow \R_+$.
It is related to the probability of finding a point at a given distance from  another point and can be defined as
$$g(x,y)=\frac{\rho^{(2)}(x,y)}{\lambda^2}\,,$$
where $\lambda=\E{\Phi\left([0,1]^d\right)}$ is the intensity of the point process and
$\rho^{(2)}$ is its {\em joint second-order intensity}; i.e.
the density (if it exists, with respect to the Lebesgue measure) 
of the second-order factorial moment measure $\alpha^{(2)}(\md (x,y))$ 
(cf. Sect.~\ref{blaszcyszyn_ss.Moments}).
%

\begin{svgraybox} 
Similarly as for Ripley's K-function, we can say that
larger values $g(x,y)$  of the pair correlation function indicate more
clustering ``around'' the vector \hbox{$x-y$}.
\end{svgraybox}

For stationary point processes the following relation holds between functions~$\rho$ and~$K$
$$K(r)=\int_{B_0(r)}g(0,y)\,\md y\,,$$
which simplifies to 
$$K(r)=d\nu_d\int_0^rs^{d-1}g(0,s)\,\md s\,,$$
in the case of isotropic processes; cf~\cite[Eq. (4.97),~(4.98)]{stoyetal95}.

For a Poisson point process $\Pi_\lambda$, we have that 
$g(x,y)\equiv 1$.
Again, it is not immediate to find examples of point processes whose pair
correlation functions are ordered for all values of $x,y$. Examples of
such point processes will be provided in the following sections.

{
\begin{exercise}
Show that ordering of pair correlation functions implies ordering of Ripley's K-functions, i.e., for two stationary point processes $\Phi_1,\Phi_2$ with $\rho_1^{(2)}(x,y) \leq \rho_2^{(2)}(x,y)$ for almost all $(x,y)\in \R^{2d}$, it holds that $K_1(r) \leq K_2(r)$ for all $r\geq 0$. 
\end{exercise}
}

{Though Ripley's K-function and the pair-correlation function are very simple to compute, they define only a pre-odering of point processes,  because their equality does not imply equality of the underlying point processes. We shall now present some possible definitions of partial ordering of point processes that capture clustering phenomena.}

\subsection{Moment Measures}
\label{blaszcyszyn_ss.Moments}
Recall that the measure  $\alpha^{k}: \mathcal{B}^{kd} \rightarrow [0,\infty]$ defined
by 
$$\alpha^{k}(B_1\times\cdots\times B_k)
=\EXP \prod_{i=1}^k \Phi(B_i) $$ 
for all (not
necessarily disjoint) bounded Borel sets $B_i$  ($i=1,\ldots, k$)
is called the {\em $k\,$-th order moment
measure} of $\Phi$. \index{moment measure}
 For simple point processes,
the truncation of the measure $\alpha^k$ to the subset
$\{(x_1,\ldots,x_k)\in (\R^{d})^k: x_i\not=x_j,\; \text{for\;} i\not=j\}$
is equal to the {\em $k\,$-th order factorial moment
measure} $\alpha^{(k)}$. \index{moment measure!factorial}
Note that  $\alpha^{(k)}(B\times\cdots\times B)$ expresses the
expected number of $k$-tuples of points of the point process in a
given set $B$. 

\begin{svgraybox} 
In the
class of point processes with some given intensity measure $\alpha=\alpha^1$,
larger values $\alpha^k(B)$ and $\alpha^{(k)}(B)$ of the
(factorial) moment  measures $\alpha^{k}$ and
$\alpha^{(k)}$, respectively, indicate point processes clustering more in $B\subset\R^d$.
\end{svgraybox}
A first argument we can give to support the above statement is
  considered in Exercise~\ref{Blasz.ex.n} below.

%
{
\begin{exercise}\label{Blasz.ex.n}
Show that comparability of $\alpha^{(2)}(B)$ 
for all bounded Borel sets $B$  implies a corresponding inequality for the  pair correlation functions
and hence Ripley's K-functions.
\end{exercise}
}

\begin{remark}
\label{blaszcyszyn_r.moments}
{For a stronger justification of the relationship between moment measures and clustering,} we can use concentration inequalities,
\index{point process!concentration inequality}
 which give upper bounds on the 
probability that the random counting measure $\Phi$ deviates from its
intensity measure $\alpha$.

Smaller deviations
can be interpreted as supportive for spatial homogeneity.
To be more specific,  using Chebyshev's inequality we have 
$$\P \left( |\Phi(B)-\alpha(B)|\ge a \right) \le
(\alpha^2(B)-(\alpha(B))^2)/a^2$$ for all bounded Borel sets $B$, and
$a> 0$.  
Thus, for point processes of the same mean measure,
the second moments or the Ripley's
functions (via  Proposition \ref{blaszcyszyn_prop:stoyan1983}) allow to compare
their clustering.
Similarly, using  Chernoff's bound, we get that
\begin{equation}
\label{blaszcyszyn_e.LT+exp}
\P \left( \Phi(B)-\alpha(B) \ge a \right) \le 
e^{-t(\alpha(B)+a)}\EXP{e^{t\Phi(B)}}
=e^{-t(\alpha(B)+a)}\sum_{k=0}^\infty \frac{t^k}{k!}\alpha^k(B)
\end{equation}
for any $t,a>0$.
Both concentration inequalities 
give smaller upper bounds  for the probability of the 
deviation from the mean (the
upper deviation in the case of Chernoff's bound)
for point processes clustering less in the sense of higher-order
moment measures. We will come back to this idea in 
Propositions~\ref{blaszcyszyn_p.moments_Laplace} and~\ref{blaszcyszyn_prop:conc_ineq} below.
\end{remark}

In Sect.~\ref{blaszcyszyn_sec:results} we will present results, in particular regarding percolation properties of point processes, for which it its enough to be able to compare factorial moment measures of point processes. We shall note casually that restricted to a "nice" class of point processes, the factorial moment measures uniquely determine the point process and hence the ordering defined via comparison of factorial moment measures is actually a partial order on this nice class of point processes.

We now concentrate on comparison to the Poisson point process.
Recall that for a general Poisson point process $\Pi_\Lambda$ we have
$\alpha^{(k)}(\md (x_1,\ldots,x_k))=\Lambda(\md x_1)\ldots\Lambda(\md x_k)$
for all $k\ge1$, where  $\Lambda=\alpha$ is the intensity measure $\Pi_\Lambda$.
In this regard, we define the following  class of point processes clustering
less (or more) than the Poisson point process with the same intensity measure. 

\begin{definition}[$\alpha$-weakly sub-Poisson point process]
\index{point process!sub-Poisson!alpha-weakly@$\alpha$-weakly}
\index{point process!super-Poisson!alpha-weakly@$\alpha$-weakly}
A point process $\Phi$ is said to be {\em weakly sub-Poisson} in the sense of moment measures
({$\alpha$-weakly sub-Poisson} for short) if
\begin{equation}\label{blaszcyszyn_e.alpha-weakly-sub-Poisson}
\EXP \prod_{i=1}^k\Phi(B_i) \le \prod_{i=1}^k \E \Phi(B_i),
\end{equation} for all $k\geq 1$ and all mutually disjoint bounded Borel sets $B_1,\ldots,B_k\subset\R^d$.  
When the reversed inequality in~(\ref{blaszcyszyn_e.alpha-weakly-sub-Poisson}) holds, we
say that $\Phi$ is {\em weakly super-Poisson} in the sense of
moment measures ($\alpha$-weakly super-Poisson for short).
\end{definition}

In other words, $\alpha$-weakly sub-Poisson point processes have
factorial moment measures $\alpha^{(k)}$ 
smaller than those of
the Poisson point process with the same intensity measure.
Similarly, 
 $\alpha$-weakly super-Poisson point processes have
factorial moment measures larger than those of
the Poisson point process with the same intensity measure.
We also remark that the notion of  sub- and super-Poisson distributions 
is used e.g. in quantum physics and denotes distributions 
for which the variance is smaller (respectively larger) than the mean.
Our notion of $\alpha$-weak sub- and super-Poissonianity is consistent with
(and stronger than) this definition. 
In quantum optics, e.g.   sub-Poisson patterns of photons appear
in resonance fluorescence, where laser light gives Poisson statistics of
photons,  while the thermal light gives super-Poisson  patterns of
photons; cf.~\cite{Photon_antibunchin_wikipedia}.

\remove{
\begin{remark}
Recall that moment measures $\alpha^{k}:\mathcal{B}^{kd} \rightarrow [0,\infty]$ of a general
point process can be expressed as non-negative 
combinations of {products of its (lower-dimensional)}  
factorial moment measures  
(cf. \cite[Ex.~5.4.5, p.~143]{DVJI2003}). Consequently,
we get that $\alpha$-weakly sub- (super-)Poisson point processes 
have also  moment measures $\alpha^{k}(\cdot)$ 
smaller (larger) than those of the corresponding Poisson point process.
\end{remark}
}

{
\begin{exercise}
 Show that $\alpha$-weakly sub- (super-) Poisson point processes 
have moment measures $\alpha^{k}$ 
smaller (larger) than those of the corresponding Poisson point process.
\textit{Hint.} Recall that the moment measures $\alpha^{k}:\mathcal{B}^{kd} \rightarrow [0,\infty]$ of a general
point process can be expressed as non-negative 
combinations of products of its (lower-dimensional)
factorial moment measures  
(cf. \cite{DVJI2003} Exercise~5.4.5, p.~143).
\end{exercise}
}

Here is an easy, but important consequence of the latter
observation regarding  Laplace transforms ``in the negative domain'',
i.e. functionals {$\cL_\Phi(-f)$, where
$$\cL_\Phi(f)=\allowbreak
\EXP{\exp\left(-\int_{\R^d} f(x)\,\allowbreak\Phi(\md x)\right)}\,,$$} for
non-negative functions $f$ on $\R^d$, which 
include as a special case the  functional $\EXP e^{t\Phi(B)}$
appearing in the ``upper'' concentration inequality~(\ref{blaszcyszyn_e.LT+exp}).
By Taylor expansion of the 
exponential function at $0$ and the well-known expression of the Laplace
functional of the Poisson point process with intensity measure $\alpha$ which
can be recognised in the right-hand side of~(\ref{blaszcyszyn_e.moments-Laplace}),
the following result is obtained. 
\begin{proposition}
\label{blaszcyszyn_p.moments_Laplace}
Assume that $\Phi$ is a simple point process with locally bounded intensity measure $\alpha$
and consider $f\ge0$. 
If $\Phi$ is $\alpha$-weakly sub-Poisson, then 
\begin{equation}
\label{blaszcyszyn_e.moments-Laplace}
\EXP \exp\left( \int_{\R^d} f(x)\,\Phi(\md x) \right)\le 
\exp \left( \int_{\R^d}(e^{f(x)}-1)\,\alpha(\md x) \right) \,.
\end{equation}
If $\Phi$ is $\alpha$-weakly super-Poisson, then
the reversed inequality is true.   
\end{proposition}

The notion of weak sub(super)-Poissonianity is closely related to 
negative and positive association of point processes, as we shall see in
Sect.~\ref{blaszcyszyn_ss.association} below.

\subsection{Void Probabilities}
\label{blaszcyszyn_ss.Voids}
The celebrated R\'enyi theorem says that  the void
probabilities $v(B)=\P(\Phi(B)=0)$ of point processes,
\index{void probability}
evaluated for all bounded Borel Sets $B$ characterise the distribution of a
simple~point process. They also allow an easy comparison of
clustering properties of point processes by  the following interpretation:
a point process having smaller void probabilities 
has less chance to create a particular hole
(absence of points in a given region).  
\begin{svgraybox}
Larger void probabilities indicate point processes with stronger
clustering. Using void probabilities in the study of clustering  is complementary  to the comparison of moments.
\end{svgraybox}
\begin{remark}
\label{blaszcyszyn_r.moments-}
An easy 
way to see the complementarity of  voids and moment measures 
consists in using again  Chernoff's bound 
to obtain the following ``lower'' concentration inequality 
\index{point process!concentration inequality}
(cf. Remark~\ref{blaszcyszyn_r.moments})
\begin{equation}
\label{blaszcyszyn_e.LT-exp}
\P \left( \alpha(B)-\Phi(B)\ge a \right) \le
e^{t(\alpha(B)-a)}\EXP{e^{-t\Phi(B)}},
\end{equation}
which holds for any $t,a>0$, 
and  noting that 
$$\EXP e^{-t\Phi(B)} = \sum_{k=0}^\infty
e^{-tk}\P \left(\Phi(B)=k \right) =\P \left( \Phi'(B)=0 \right)=v'(B)$$
is the void probability of the point process $\Phi'$ obtained from
$\Phi$ by independent thinning with retention probability
$1-e^{-t}$. It is not difficult to show that ordering of void
probabilities of simple point processes is preserved by independent thinning
(cf.~\cite{dcx-perc})
and thus the bound in~(\ref{blaszcyszyn_e.LT-exp}) is smaller for point processes 
less clustering in the latter sense.
We will come back to this idea  in Propositions~\ref{blaszcyszyn_p.voids_Laplace}
and~\ref{blaszcyszyn_prop:conc_ineq}.
Finally, note that $\lim_{t\to\infty}\EXP e^{-t\Phi(B)}=v(B)$ and
thus, in conjunction with what was said above, 
comparison of void probabilities is equivalent to the
comparison of one-dimensional 
Laplace transforms of point processes for non-negative arguments.
\end{remark}

In Sect.~\ref{blaszcyszyn_sec:results}, we will present results, in particular regarding percolation properties, for which it is enough to be able to compare void probabilities of point processes. {Again, because of R\'{e}nyi's theorem, we have that ordering defined by void probabilities is a partial order on the space of simple point processes.}

\subsubsection{$v$-Weakly Sub(Super)-Poisson Point Processes}
\index{point process!sub-Poisson!v-weakly@$v$-weakly}
\index{point process!super-Poisson!v-weakly@$v$-weakly}
Recall that a Poisson point process $\Pi_\Lambda$ can be characterised as having void probabilities of the form
$v(B)=\exp(-\Lambda(B))$, with $\Lambda$ being the intensity
measure of~$\Phi$. 
In this regard, we define the following  classes of point processes clustering
less (or more)
than the Poisson point process with the same intensity measure. 
\begin{definition}[$v$-weakly sub(super)-Poisson point process]
A point process $\Phi$ is said to be {\em weakly sub-Poisson} in the sense of void
  probabilities ({$v$-weakly sub-Poisson} for short) if
\begin{equation}\label{blaszcyszyn_e.nu-weakly-sub-Poisson} \P\left( \Phi(B)=0 \right) \le
e^{- \E \Phi(B)}
\end{equation}
for all Borel sets $B \subset \R^d$.  When the reversed inequality
in~(\ref{blaszcyszyn_e.nu-weakly-sub-Poisson}) holds, we say that $\Phi$
is {\em weakly super-Poisson} in the sense of void probabilities
({$v$-weakly super-Poisson} for short).
\end{definition}

In other words, $v$-weakly sub-Poisson point processes have
void probabilities smaller than those of
the Poisson point process with the same intensity measure. Similarly, 
 $v$-weakly super-Poisson point processes have
void probabilities larger than those of
the Poisson point process with the same intensity measure. 

\begin{example}\label{exe.Cox}
\index{Cox process!v-weakly super-Poisson@$v$-weakly super-Poisson}
It is easy to see by Jensen's inequality that all Cox point processes are
  $v$-weakly super-Poisson.
\end{example}

By using a coupling argument as in Remark \ref{blaszcyszyn_r.moments-}, we can derive an analogous result as in Proposition~\ref{blaszcyszyn_p.moments_Laplace} for $v$-weakly sub-Poisson point processes.
\begin{proposition}[{~\upshape{\bf{\cite{dcx-perc}}}}]
\label{blaszcyszyn_p.voids_Laplace}
Assume that $\Phi$ is a simple point process with locally bounded intensity measure $\alpha$. 
Then  $\Phi$ is $v$-weakly sub-Poisson if and only if 
{ {\upshape{(\ref{blaszcyszyn_e.moments-Laplace}}}) holds
for all functions $f\ge 0$}.
\end{proposition}

\subsubsection{Combining Void Probabilities and Moment Measures}
We have already explained why the comparison of void probabilities and moment measures are
in some sense complementary.
Thus, it is {natural} to combine them, whence the following definition is obtained.
\begin{definition}[Weakly sub- and super-Poisson point process]
\index{point process!sub-Poisson!weakly}
\index{point process!super-Poisson!weakly}
We say that $\Phi$ is {\em weakly sub-Poisson}
if $\Phi$ is $\alpha$-weakly sub-Poisson and $v$-weakly sub-Poisson.
Weakly super-Poisson point processes are defined in the same way.
\end{definition}

The following remark is immediately obtained from 
Propositions~\ref{blaszcyszyn_p.moments_Laplace} and~\ref{blaszcyszyn_p.voids_Laplace}.
\begin{remark}
\label{blaszcyszyn_p.associated_Laplace_remark}
Assume that $\Phi$ is a simple point process with locally bounded intensity measure~$\alpha$.
If $\Phi$ is weakly sub-Poisson then~(\ref{blaszcyszyn_e.moments-Laplace})
holds for any $f$ of constant sign ($f\ge0$ or $f\le0$).
If $\Phi$ is weakly super-Poisson, 
then in (\ref{blaszcyszyn_e.moments-Laplace}) the reversed inequality holds for such $f$.
\end{remark}

\begin{example}
\label{blaszcyszyn_ex.Det-Perm-weakly}
\index{determinantal process!weakly sub-Poisson}
\index{permanental process!weakly super-Poisson}
It has been shown in~\cite{dcx-clust} that determinantal and permanental 
point process {(with trace-class integral
  kernels)} are weakly sub-Poisson and weakly super-Poisson,
respectively.  
\end{example}

Other examples (admitting even stronger comparison properties) will be given 
in Sect.~\ref{blaszcyszyn_ss.association} and~\ref{blaszcyszyn_ss.dcx}.

As mentioned earlier, using the ordering of Laplace functionals of weakly sub-Poisson point processes, we can extend the concentration inequality for Poisson 
point processes to this class of point processes.  In the discrete setting, a similar result is proved for negatively associated random variables in
\cite{Dubhashi96}. A more general concentration inequality for Lipschitz functions is known in the case of determinantal point processes (\cite{Pemantle11}).   
\index{point process!sub-Poisson!concentration inequality}
\index{point process!concentration inequality}


\begin{proposition}
\label{blaszcyszyn_prop:conc_ineq}
Let $\Phi$ be a simple stationary point process with unit intensity which is weakly
sub-Poisson, and let $B_n\subseteq\R^d$ be a Borel set of Lebesgue measure
$n$. Then, for any $1/2<a<1$ there exists an integer $n(a)\geq 1$ such that for $n\ge n(a)$
$$\P(|\Phi(B_n) - n| \geq n^{a} ) \leq 2\exp \left(-n^{2a-1}/9\right).$$
\end{proposition}
\remove{
\begin{proof}
By using Markov's inequality along with Propositions~\ref{blaszcyszyn_p.moments_Laplace} and ~\ref{blaszcyszyn_p.voids_Laplace}, we have the following two inequalities. 
It holds that 
\begin{eqnarray*}
\P(\Phi(B_n) \geq n + n^a) & \leq & z^{-n - n^a}\E(z^{\Phi(B_n)}) \leq  z^{-n-n^a}e^{n(z-1)}, \, \, \, z \geq 1. \\
\P(\Phi(B_n) \geq n - n^a) & \leq & z^{-n + n^a}\E(z^{\Phi(B_n)}) \leq  z^{-n+n^a}e^{n(z-1)}, \, \, \, z \leq 1.
\end{eqnarray*}
Putting $z= (n+n^a)/n$ and  $z= (n-n^a)/n$, respectively, the right-hand sides of
the above inequalities  can be upper-bound by $\exp\left[-n^{2a-1}/9\right]$ for sufficiently large $n$
(cf. \cite[Lemmas 1.2 and 1.4]{Penrose03}).
This completes the proof.
\qed
\end{proof}
}

{
\begin{exercise}
Prove Proposition~\ref{blaszcyszyn_prop:conc_ineq}. \textit{Hint.} Use Markov's inequality, Propositions~\ref{blaszcyszyn_p.moments_Laplace} and ~\ref{blaszcyszyn_p.voids_Laplace} along with the bounds for the Poisson case known from \cite[Lemmas 1.2 and 1.4]{Penrose03}. 
\end{exercise}
}

Note that the bounds we have suggested to use are the ones corresponding to the Poisson point process. For specific weakly sub-Poisson point processes, one expects an improvement on these bounds.


\subsection{Positive and Negative Association}
\label{blaszcyszyn_ss.association}
Denote covariance of (real-valued) 
random variables $X,Y$ by $\COV{X,Y}=\EXP{XY}-\E X\E Y$.
\begin{definition}[(Positive) association of point processes]
\index{point process!associated|seealso{association}}
\index{point process!associated}
\label{blaszcyszyn_d.p-association}
A point process  $\Phi$ 
is called  {\em associated} if 
\begin{equation}\label{blaszcyszyn_eqn:association}
\COV{f(\Phi(B_1),\ldots,\Phi(B_k)),g(\Phi(B_1),\ldots,\Phi(B_k))}\allowbreak
\geq 0
\end{equation}
for any finite collection of bounded Borel sets $B_1,\ldots,B_k\subset\R^d$ and
$f,g : \R^k \to [0,1]$ (componentwise) increasing functions; cf.~\cite{BurtonWaymire1985}. 
\end{definition}

The property considered in~\eqref{blaszcyszyn_eqn:association} is also called {\em
  positive association}, or the {\em FKG property}.
The theory for the opposite property is more tricky, cf.~\cite{pemantle00},
but one can define it as follows.
\begin{definition}[Negative association]
\index{point process!negatively associated|seealso{negative association}}
\index{point process!negatively associated}
\index{negative association}
\label{blaszcyszyn_d.n-association}
A point process $\Phi$ is called {\em negatively associated} if
$$\COV{f(\Phi(B_1),\ldots,\Phi(B_k)),\allowbreak g(\Phi(B_{k+1}),\ldots,\Phi(B_l))}\allowbreak \le 0$$
for any finite collection of bounded Borel sets $B_1,\ldots,B_l\subset\R^d$ 
such that $(B_1\cup\dots\cup  B_k)\cap(B_{k+1}\cup\dots \cup
B_{l})=\emptyset$ and $f,g$ increasing functions.
\end{definition}

Both definitions can be straightforwardly extended to
arbitrary random measures, where one additionally assumes that $f,g$
are continuous and increasing functions. Note that the notion of
association or negative association of point processes does not induce
any ordering on the space of point processes. Though, association or negative association have not been studied from the point of view of stochastic ordering, it has been widely used to prove central limit theorems for random fields (see \cite{Bulinski_Spodarev2012central}).

\begin{svgraybox}
Positive and negative association can be seen as clustering comparison
to Poisson point process.
\end{svgraybox}

The following result, supporting the above statement,
has been proved in~\cite{dcx-clust}. It will be strengthened in the next section (see Proposition~\ref{blaszcyszyn_newProp})
\begin{proposition}
\label{blaszcyszyn_p.negassoc-weak}
\index{negative association}
\index{point process!sub-Poisson!weakly}
\index{association}
\index{point process!super-Poisson!weakly}
A negatively associated, simple point process with locally bounded intensity measure 
is weakly sub-Poisson.
A (positively) associated point process with a diffuse locally bounded intensity measure 
is weakly super-Poisson.
\end{proposition}

{
\begin{exercise}
Prove that a (positively) associated point process with a diffuse locally bounded intensity measure is $\alpha$-weakly super-Poisson. Show a similar statement for negatively associated point processes as well.
\end{exercise}
}

\begin{example}
\index{Poisson cluster process!association}
From~\cite[Th.~5.2]{BurtonWaymire1985}, we know that any {\em
  Poisson cluster point process} is associated.
 This is a generalisation of the 
perturbation approach of a Poisson point process $\Phi$ considered in~(\ref{blaszcyszyn_defn:perturbed_pp})
having the form $\Phi^{\text{cluster}}=\sum_{X\in\Phi}(X+\Phi_X)$ with $\Phi_X$ being
arbitrary i.i.d. (cluster) point processes. 
In particular, the 
\index{Neyman-Scott process!association}
Neyman-Scott point process (cf. Example~\ref{blaszcyszyn_ex.NS-PP})
is associated. Other examples of
associated point processes given in~\cite{BurtonWaymire1985} are Cox point processes with random intensity measures being associated.
\end{example}

\remove{
\begin{example}
\index{association}

The binomial point process (cf. Example~\ref{blaszcyszyn_ex.Binomial}) is negatively
associated.
\end{example}
}
{

\begin{example}
\label{DPP_assoc}
Determinantal point processes are negatively associated (see \cite[cf. Corollary 6.3]{Ghosh12a}).
\end{example}

We also remark that there are negatively associated point
processes, which are not weakly sub-Poisson. 
A counterexample
given in~\cite{dcx-clust} (which is not a simple point process,
showing that  this
latter assumption cannot be relaxed
in~Proposition~\ref{blaszcyszyn_p.negassoc-weak}) exploits  
\cite[Theorem~2]{JoagDev1983}, which says that 
a random vector having a {\em  permutation distribution} 
(taking as values all $k!$ permutations
of a given deterministic vector with equal probabilities)
is negatively associated.

%

{
\begin{exercise}
\index{binomial process!association}
Show that the binomial point process (cf. Example~\ref{blaszcyszyn_ex.Binomial}) and the simple perturbed lattice 
(cf. Example~\ref{blaszcyszyn_ex:Voronoi_pert_lattice}) are negatively associated.
\end{exercise}
}

\subsection{Directionally Convex Ordering}
\label{blaszcyszyn_ss.dcx}

\subsubsection{Definitions and Basic Results}
In this section, we present some basic results on directionally convex ordering of
point processes that will allow us to see this order also as a tool to
compare clustering of point processes. 

A Borel-measurable function $f:\R^k \rar \R$ is said to be {\em
  directionally convex}~(dcx) if for any $x \in \R^k,
\epsilon,\delta > 0, i,j \in \{1,\ldots,k\}$, we have that
$\Delta_{\epsilon}^i \Delta_{\delta}^jf(x) \geq 0$, where
$\Delta_{\epsilon}^if(x) = f(x + \epsilon e_i) - f(x)$
is the discrete differential
operator, with  $\{e_i\}_{1 \leq i \leq k}$
denoting  the canonical basis vectors of $\R^k$. 
In the following, we abbreviate {\em  increasing} and dcx by idcx and {\em
decreasing} and dcx by ddcx (see \cite[Chap. 3]{Muller02}). 
For random vectors $X$ and $Y$ of the
same dimension, {\em $X$ is said to be smaller than $Y$ in
dcx order} (denoted  $X \leq_{\text{dcx}} Y$)
if $\E f(X) \leq \E f(Y)$ for all $f$ being dcx
such that both expectations in the latter inequality are finite. 
Real-valued random fields are said to be dcx ordered if all
finite-dimensional marginals are dcx ordered.
\begin{definition}[dcx-order of point processes]
\index{point process!dcx-order@dcx order|seealso{dcx order}}
\index{point process!dcx-order@dcx order}
\index{dcx order@dcx order}
Two point processes $\Phi_1$ and $\Phi_2$ are said to be dcx-ordered, i.e. $\Phi_1 \leq_{\text{dcx}} \Phi_2$, if for any $k\geq 1$ and bounded Borel sets $B_1, \ldots, B_k$ in $\R^d$,
it holds that $(\Phi_1(B_1),\ldots,\Phi_1(B_k)) \leq_{\text{dcx}}
(\Phi_2(B_1),\ldots,\Phi_2(B_k))$.
\end{definition}

The definition of comparability of point processes is similar for other orders, i.e. those defined by
$\text{idcx},\allowbreak \text{ddcx}$ functions. It is enough to verify the above
conditions for $B_1,\ldots,B_k$ mutually disjoint, cf.~\cite{snorder}. 
In order to avoid technical difficulties, we will
consider only point processes whose intensity measures
are locally finite. For such point processes, the dcx-order is a
{partial} order. 

\begin{remark}
It is easy to see that 
$\Phi_1 \leq_{\text{dcx}} \Phi_2$ {\em implies the  equality
of their  intensity measures}, i.e: $\E\Phi_1(B) =\E\Phi_2(B)=\alpha(B)$
for any bounded Borel set $B\subset \R^d$ as both $x$ and $-x$ are dcx functions. 
\end{remark}

We argue that, dcx-ordering is also useful in 
clustering comparison of point processes.
\begin{svgraybox}
Point processes larger in dcx-order cluster
more, whereas point processes larger in idcx-order cluster
more while having on average more points, and point processes 
larger in ddcx-order cluster more while having on average less points.
\end{svgraybox}

The two statements of the following result were proved 
in~\cite{snorder} and~\cite{dcx-clust}, respectively. They show 
that dcx-ordering is stronger than
comparison of moments measures and void probabilities
considered in the two previous sections.
\begin{proposition}
\label{blaszcyszyn_Prop:dcx-alpha-nu}
\index{void probability}
\index{moment measure}
\index{dcx order@dcx order}
Let $\Phi_1$ and $\Phi_2$ be two point process on $\R^d$. Denote their 
moment measures by $\alpha_j^{k}$ ($k\ge1$)
 and their void probabilities by $v_j$, $j=1,2$, respectively.
\begin{enumerate}
\item If $\Phi_1 \leq_{\rm{idcx}} \Phi_2$ then
 $\alpha_1^{k}(B)\le \alpha_2^{k}(B)$ for all bounded Borel sets $B\subset
 (\R^{d})^k$,  provided that $\alpha_j^{k}$ is $\sigma$-finite for 
$k\ge1$, $j=1,2$. 
\item If  $\Phi_1 \leq_{\text{\rm{ddcx}}} \Phi_2$ then
$v_1(B) \le v_2(B)$
for all bounded Borel sets $B\subset\R^d$.
\end{enumerate}
\end{proposition}
{
\begin{exercise}
Show that $\prod_i(x_i \vee 0)$ is a dcx-function and $(1-x) \vee 0$ is a convex function. Using these facts to prove the above proposition. 
\end{exercise}
}
Note that the $\sigma$-finiteness condition considered in the first statement of Proposition~\ref{blaszcyszyn_Prop:dcx-alpha-nu} is missing
  in~\cite{snorder}; see~\cite[Proposition~4.2.4]{Yogesh_thesis} for the
  correction.
An important observation is that the  operation of clustering
perturbation introduced in Sect.~\ref{blaszcyszyn_ss.perturbation}
 is dcx monotone with respect to the replication kernel in the
 following sense; cf.~\cite{dcx-clust}.
\begin{proposition}\label{blaszcyszyn_p.pert-lattice}
\index{perturbation kernel!convex ordering}
\index{dcx order@dcx order!of perturbed processes}
Consider a point process $\Phi$ with locally finite intensity measure $\al$
and its two perturbations
$\Phi_j^{\rm{pert}}$ ($j=1,2$) satisfying
condition~({\upshape{\ref{blaszcyszyn_e:perturbation-Radon}}}), and having the same
displacement kernel $\cX$ and possibly different replication kernels
$\cN_j$, $j=1,2$, respectively.
If $\cN_1(x,\cdot)\leq_{\rm{cx}}\cN_2(x,\cdot)$ (which means convex
ordering of the conditional distributions of the number of replicas)
for $\al$-almost all $x\in\R^d,$ then
$\Phi^{\rm{pert}}_1\leq _{\rm{dcx}} \Phi^{\rm{pert}}_2$.
\end{proposition}

Thus clustering perturbations of a given point process provide
many examples of point process comparable in dcx-order. Examples 
of convexly ordered replication kernels have been given in  
Example~\ref{blaszcyszyn_ex:Voronoi_pert_lattice}.

Another  observation, proved in~\cite{snorder},
says that the  operations transforming some  random 
measure $\rndLambda$ into a Cox point process  $\Cox_{\rndLambda}$ 
(cf. Definition~\ref{blaszcyszyn_d.Cox}) preserves the dcx-order.
\begin{proposition}
\index{Cox process!dcx-ordering@dcx ordering}
\index{dcx order@dcx order!of Cox processes}
\label{blaszcyszyn_prop:int_fld_meas}
Consider two random measures  $\rndLambda_1$ and $\rndLambda_2$ 
on $\R^d$.
If $\rndLambda_1 \leq_{\rm{dcx\resp{\rm{idcx}}}}
\rndLambda_2$ then $\rm{\Cox}_{\rndLambda_1} \leq_{\rm{dcx\resp{\rm{idcx}}}} \rm{\Cox}_{\rndLambda_2}$.
\end{proposition}

The above result, combined with further results on comparison of
shot-noise fields presented in Sect.~\ref{blaszcyszyn_sss.SN}
will allow us to compare many Cox point processes (cf. Example~\ref{blaszcyszyn_ex.dcx-Cox}).

\subsubsection{Comparison of Shot-Noise Fields}
\label{blaszcyszyn_sss.SN}
Many interesting quantities in stochastic geometry can be expressed
by additive or extremal shot-noise fields. They are also used
to construct more sophisticated point process models.
For this reason, we state some results on dcx-ordering of shot-noise fields that are widely used in applications. 
\begin{definition}[Shot-noise fields]
\label{blaszcyszyn_defn:ISN}
\index{shot-noise field}
Let $S$ be any (non-empty) index set.
Given a point process
$\Phi$ on $\R^d$ and a {\em response function} $h(x,y) :\R^d \times S \rightarrow (-\infty, \infty]$ which is  measurable in the first variable, then
the {\em (integral) shot-noise field} $\{ V_{\Phi}(s), s\in S\}$ is
defined as
\begin{equation}
\label{blaszcyszyn_eqn:sn_rm}
 V_{\Phi}(y) = \int_{\R^d}h(x,y)\Phi(\md x) = \sum_{X \in \Phi}h(X,y),
\end{equation}
and the {\em extremal shot-noise} field $\{U_{\Phi}(s), s\in S\}$ is defined as
\index{shot-noise field!extremal}
\begin{equation}
\label{blaszcyszyn_eqn:esn}
U_{\Phi}(y) = \sup_{X \in \Phi}\{h(X,y)\}.
\end{equation}
\end{definition}

As we shall see in Sect.~\ref{blaszcyszyn_sec:u_stat}
(and also in the proof of Proposition~\ref{blaszcyszyn_prop:ord_cap_fnl})
 it is not merely a formal generalisation to take $S$ being an arbitrary set. Since the  composition of a dcx-function with an  increasing linear function is still dcx, 
linear combinations of $\Phi(B_1),\ldots, \Phi(B_n)$ for finitely many bounded Borel sets $B_1,\ldots, B_n\subseteq\R^d$ (i.e. $\sum_{i=1}^mc_i\Phi(B_i)$ for $c_i \geq 0$) preserve the dcx-order. An integral shot-noise field can be approximated by finite linear combinations of $\Phi(B)$'s and hence justifying continuity, one expects that integral shot-noise fields preserve dcx-order as well. This type of important results on dcx-ordering of point processes is stated below.
\begin{proposition}({~\upshape{\bf{\cite[Theorem 2.1]{snorder}}}})
\label{blaszcyszyn_thm:isn_rm}
\index{dcx order@dcx order!of shot-noise fields}
\index{shot-noise field!dcx order@dcx order of }
Let $\Phi_1$ and $\Phi_2$ be arbitrary point processes on $\R^d$. 
Then, the following statements are true.
\begin{enumerate}
\item If $\Phi_1  \leq_{\rm{idcx}} \Phi_2$,
then $\{V_{\Phi_1}(s), s\in S\} \leq_{\rm{idcx}} \{V_{\Phi_2}(s), s\in S\}$.

\item If $\Phi_1  \leq_{\rm{dcx}} \Phi_2$,
then $\{V_{\Phi_1}(s), s\in S\} \leq_{\rm{dcx}} \{V_{\Phi_2}(s), s\in S\}$, provided that $\E V_{\Phi_i}(s) < \infty$,  for all $s \in S$,
  $i=1,2$. 
\end{enumerate}
\end{proposition}

The results of Proposition~\ref{blaszcyszyn_thm:isn_rm}, combined with 
those of Proposition~\ref{blaszcyszyn_prop:int_fld_meas}
allow the comparison of many Cox processes.
\begin{example}[Comparable Cox point processes]
\index{Cox process!dcx-ordering@dcx ordering}
\index{dcx order@dcx order!of Cox processes}
\index{dcx order@dcx order!of Levy-based Cox processes@of L\'evy
  based Cox processes}
\label{blaszcyszyn_ex.dcx-Cox}
Let $Z_1$ and $Z_2$ be two L\'{e}vy-bases with mean measures
$\alpha_1$ and $\alpha_2$, respectively. Note that $\alpha_i \leq_{\text{dcx}} Z_i$ ($i=1,2$). This can be easily proved using complete independence of L\'evy bases and Jensen's inequality.
In a sense, the mean measure $\alpha_i$ ``spreads'' (in the sense of dcx) the 
mass better than the corresponding completely independent random measure $Z_i$. 
Furthermore, consider the random fields $\rndlambda_1$ and $\rndlambda_2$ on $\R^d$ given by $\rndlambda_i(y) =\int_{\R^d} k(x,y)Z_i(\md x)$, $i=1,2$
for some non-negative kernel $k$, and assume that these fields
are a.s. locally Riemann integrable.
Denote by $\Cox_{\rndlambda_i}$  and $\Cox_{\exp(\rndlambda_i)}$ the 
corresponding L\'{e}vy-based  and log-L\'{e}vy-based  Cox point process.
The following inequalities hold.
\begin{enumerate}
\item If $Z_1 \leq_{\text{dcx\resp{\text{idcx}}}} Z_2$, then $\Cox_{\rndlambda_1}
  \leq_{\text{dcx\resp{idcx}}} \Cox_{\rndlambda_2}$
 provided that, in case of dcx,  $\EXP{\int_B\rndlambda_i(y)\,\md
   y}<\infty$ for all bounded Borel sets $B\subset \R^d$. 
\item If $Z_1 \leq_{\text{idcx}} Z_2$, then  $\Cox_{\exp(\rndlambda_1)}
  \leq_{\text{idcx}} \Cox_{\exp(\rndlambda_2)}$.
\end{enumerate}

Suppose that $\{X_i(y)\}, i=1,2$ are two Gaussian random fields on $\R^d$
and denote by   $\Cox_{\exp(X_i)}$, ($i=1,2$) the corresponding
log-L\'{e}vy-based  Cox point processes. Then the following is true.
\begin{enumerate}
\addtocounter{enumi}{2}
\item If 
$\{X_1(y)\} \leq_{\rm{idcx}} \{X_2(y)\}$   (as random fields), then $\Cox_{\exp(X_1)}
  \leq_{\rm{idcx}} \Cox_{\exp(X_2)}$.
\end{enumerate}

Note that the condition in the third statement is equivalent to  $\E X_1(y) \leq
\E X_2(y)$ for all $y \in \R^d$ and $\cov(X_1(y_1),X_1(y_2))
\leq \cov(X_2(y_1),X_2(y_2))$ for all $y_1,y_2 \in \R^d$. An example of a 
parametric dcx-ordered family of Gaussian random fields is given
in~\cite{miyoshi04}.

Let $\Phi_1$, $\Phi_2$ be two point processes on $\R^d$ and denote by $\Cox_1$, $\Cox_2$
the generalised shot-noise Cox point processes (cf. Example~\ref{blaszcyszyn_ex.GSNCox}) being clustering
perturbations of  $\Phi_1,$ $\Phi_2$, respectively, with the same
 (Poisson) replication kernel $\cN$ and with displacement distributions  
 $\cX(x,\cdot)$ having density $\cX'(x,y)\,\md y$ for all $x\in\R^d$. 
Then, the following result is true.
\begin{enumerate}
\addtocounter{enumi}{3}
\item  If
$\Phi_1\leq_{\text{dcx\resp{idcx}}}\Phi_2$, then $\Cox_1\le_{\text{dcx\resp{idcx}}}\Cox_2$
 provided that, in case of dcx,  $\int_{\R^d}\cX'(x,y)\,\alpha(\md x)<\infty$
for all $y\in\R^d$, where $\alpha$ is the (common) intensity measure of $\Phi_1$ and $\Phi_2$.
\end{enumerate}
\end{example}

Proposition~\ref{blaszcyszyn_thm:isn_rm} allows us to compare
extremal shot-noise fields using the following well-known representation %
$\P(U(y_i) \leq a_i , 1 \leq i \leq m) = 
\E  e^{- \sum_i \hat{U}_i}$
%
where $\hat{U}_i = \sum_n-\log{\ind(h(X_n,y_i) \leq a_i)}$ is an additive shot-noise field with response function taking values in $[0,\infty].$ Noting that $e^{-\sum_i x_i}$ is a dcx-function, we get the following result.
\begin{proposition}[{~\upshape{\bf{\cite[Proposition 4.1]{snorder}}}}]
\index{shot-noise field!extremal!ordering}
\label{blaszcyszyn_prop:lo_msn}
 Let $\Phi_1 \leq_{\rm{dcx}} \Phi_2$. Then for any $n \geq 1$ and for all $t_i \in \R, y_i \in S, 1 \leq i \leq n$, it holds that 
$$\P(U_{\Phi_1}(y_i) \leq t_i , 1 \leq i \leq n) \le \P(U_{\Phi_2}(y_i) \leq t_i, 1 \leq i \leq n).$$
\end{proposition}

An example of application of the above result is the comparison of
{\em capacity functionals} of {\em Boolean models} whose definition 
we recall first.
%

\begin{definition}[Boolean model]
\index{Boolean model}
\index{Boolean model!coverage}
\label{blaszcyszyn_d.Boolean_model}
Given  (the distribution of) a random closed set $Y$
 and a point process  $\Phi$, a {\em
  Boolean model} with the point process of {\em germs} $\Phi$  and the
{\em typical grain} $Y$, is given by the random set $C(\Phi,Y) =\bigcup_{X_i\in\Phi} \{X_i+Y_i\}$, where
$x+A=\{x+a: a\in A\}$, $a\in\R^d$,  $A\subset\R^d$ and $\{Y_i\}$ is a sequence of i.i.d.  random closed sets distributed as $Y$.
We call $Y$ a {\em fixed grain} if there exists a (deterministic) closed set $B\subseteq\R^d$ such
that $Y = B$ a.s. In the case of spherical grains, i.e. $B = B_o(r)$, where
$o$ is the origin of $\R^d$ and $r \ge 0$ a constant, we denote
the corresponding Boolean model by $C(\Phi,r)$.
\end{definition}

A commonly made technical assumption about the distributions of $\Phi$
and $Y$ is that for any compact set $K\subset\R^d$, the expected number of germs $X_i\in\Phi$ such that $(X_i+Y_i)\cap K\not=\emptyset$ is finite.
This assumption, called ``local finiteness of the Boolean
model'' guarantees in particular that $C(\Phi,Y)$ is a random closed set in $\R^d$.
The Boolean models considered throughout this chapter will be assumed to have the local finiteness property.

\begin{proposition}[{\upshape{\bf{\cite[Propostion 3.4]{perc-dcx}}}}]
\label{blaszcyszyn_prop:ord_cap_fnl}
Let  $C(\Phi_j,Y)$, $j=1,2$ be two Boolean models with point processes of germs
$\Phi_j$, $j=1,2$, respectively, and common distribution of the typical
grain $Y$. Assume that $\Phi_1$ and $\Phi_2$ are simple and have locally finite moment
measures. If $\Phi_1\le_{\rm{dcx}}\Phi_2$, then

$$\P\left(C(\Phi_1,Y) \cap B=\emptyset\right)\le\P\left(C(\Phi_2,Y) \cap  B=\emptyset\right)$$

for all bounded Borel sets $B\subset\R^d$.
Moreover, if $Y$ is a fixed compact grain, then the same result holds,
provided $v_1(B)\le v_2(B)$  for all bounded Borel sets $B\subset\R^d$, where $v_i(B)$ denotes 
the void probabilities of $\Phi_i$.
\end{proposition}

\subsubsection{Sub- and Super-Poisson Point Processes}
\label{blaszcyszyn_sss.sub-suuper-Poisson}
We now concentrate on dcx-comparison to the Poisson point process.
To this end, we define the following classes of point processes.
\begin{definition}[Sub- and super- Poisson point process]
\index{point process!sub-Poisson!dcx@dcx}
\index{point process!super-Poisson!dcx@dcx}
\index{dcx order@dcx order!sub-Poisson process}
\index{dcx order@dcx order!super-Poisson process}
We call a point process
{\em dcx sub-Poisson} (respectively {\em
 dcx  super-Poisson}) if it is smaller (larger) in dcx-order than the
Poisson point process (necessarily of the same mean measure).
For simplicity, we will just refer to them as sub-Poisson or super-Poisson point process
omitting the phrase dcx. 
\end{definition}

\begin{proposition}\label{blaszcyszyn_newProp}
 A negatively associated point processes $\Phi$ with convexly sub-Poisson
 one-dimensional marginal distributions,
 $\Phi(B)\le_{cx}\Pois(\EXP{\Phi(B)})$ for all 
 bounded Borel sets $B$, is sub-Poisson.  An associated point
 processes with convexly super-Poisson one-dimensional marginal
 distributions is 
 super-Poisson. 
\end{proposition}
\begin{proof}[sketch]
This is a consequence of~\cite[Theorem~1]{christofides2004connection},
which says  that  
a negatively associated random vector is  supermodularly smaller than
the random vector with the same marginal distributions and
independent components. Similarly, an associated random vector is
supermodularly larger than the random vector with the same marginal distributions and
independent components.
Since supermodualr  order is stronger than $dcx$ order, this implies
dcx ordering as well. Finally, a vector with independent coordinates
and  convexly sub-Poisson (super-Poisson) 
marginal distributions is $dxc$ smaller (larger) than 
the vector of independent Poisson variables.
\end{proof}

\begin{example}[Super-Poisson Cox point process]
\index{Cox process!dcx super-Poisson@dcx super-Poisson}
\label{blaszcyszyn_ex.super-Poisson-Cox}
Using Proposition~\ref{blaszcyszyn_prop:int_fld_meas} one can prove (cf.~\cite{snorder})
that Poisson-Poisson cluster point processes and, more generally,
L\'{e}vy-based Cox point processes are super-Poisson.

Also, since any mixture of Poisson distributions 
is cx larger than the Poisson distribution (with the same mean), we
can prove that any mixed Poisson point process is super-Poisson.
\end{example}

\begin{example}[Super-Poisson Neyman-Scott point process]
\index{Neyman-Scott process!dcx super-Poisson@dcx super-Poisson}
By~Proposition~\ref{blaszcyszyn_p.pert-lattice},
any  Ney\-man-Scott point process (cf.
Example~\ref{blaszcyszyn_ex.NS-PP}) with  mean cluster size $n(x)=1$ for all $x\in\R^d$ 
is super-Poisson. Indeed, for any $x\in\R^d$ and 
any replication kernel  $\cN$ satisfying
$\sum_{k=1}^\infty k\cN(x,\{k\})=1$, we have  by Jensen's
inequality that $\delta_1(\cdot)\le_{\text{cx}}\cN(x,\cdot)$, i.e. it is convexly larger
than the Dirac measure on $\Z_+$ concentrated at~1. By the well-known
displacement theorem for Poisson point processes, the clustering 
perturbation of the Poisson (parent) point process with this Dirac replication
kernel is a Poisson point process. 
Using kernels of the form mentioned in~(\ref{blaszcyszyn_eqn:super_Poisson_rvs}) we can
construct dcx-increasing super-Poisson point processes. 
\end{example}

\begin{example}[Sub- and super-Poisson perturbed lattices]
\index{lattice process!super-Poisson perturbed}
\index{lattice process!sub-Poisson perturbed}
\index{dcx order@dcx order!super-Poisson perturbed lattice}
\index{dcx order@dcx order!sub-Poisson perturbed lattice}
\label{blaszcyszyn_ex.pert_lattice_dcx}
Lattice clustering perturbations provide examples of both sub- and super-Poisson
point process, cf. Example~\ref{blaszcyszyn_ex:Voronoi_pert_lattice}. Moreover, the initial
lattice can be 
replaced by any fixed pattern of points, and the displacement kernel
needs not to be supported by the Voronoi cell of the given
point. Assuming Poisson replication kernels we still obtain (not
necessarily homogeneous) Poisson point processes.
Note, for example, that by~(\ref{blaszcyszyn_eqn:sub_Poisson_rvs})
considering binomial replication kernels $\Binom(r,\lambda/r)$
for $r\in\Z^+$, $r\ge\lambda$
 one can construct  dcx-increasing families of
sub-Poisson perturbed lattices converging to the Poisson point process $\Pi_\lambda$.
Similarly,  considering  negative binomial 
replication kernels $\NBinom(r,\lambda/(r+\lambda)$ with $r\in\Z_+$, $r\ge1$ one can construct
dcx-decreasing families of super-Poisson perturbed lattices converging to  $\Pi_\lambda$. The 
simple perturbed lattice (with $\Binom(1,1)$, and necessarily $\lambda=1$)
is the smallest point process in dcx-order within the aforementioned  sub-Poisson family. 
\end{example}

\begin{example}[determinantal and permanental processes]
\index{determinantal process! dcx}
We already mentioned in Example~\ref{blaszcyszyn_ex.Det-Perm-weakly}
that determinantal and permanental point processes are weakly sub-
and super-Poisson point processes, respectively.
Since determinantal point processes are negatively associated (Example
\ref{DPP_assoc}) and also they have convexly sub-Poisson
one-dimensional marginal distributions, cf~\cite[proof of
  Prop.~5.2]{dcx-clust}, Proposition \ref{blaszcyszyn_newProp} gives
us that determinantal point processes are $dcx$ sub-Poisson.
The statement for permanental processes 
can be strengthened to dcx-comparison to the Poisson point process
with the same mean on {\em mutually disjoint,
simultaneously observable compact subsets} of $\R^d$;
see \cite{dcx-clust} for further details on the result and
its proof.
\end{example}

\begin{figure*}[!h]
\begin{center}
\tikzstyle{method} = [draw, rectangle,fill=gray!30, node distance=13ex,
    minimum height=2em, text width=0.3\linewidth, text badly centered]
\tikzstyle{assoc} = [draw, rectangle,rounded corners, dashed, fill=gray!30, node distance=13ex, minimum height=2em, text width=0.3\linewidth, text badly centered]
\tikzstyle{line} = [draw, thick, -latex']
\begin{tikzpicture}
   \node [method,text width=0.4\linewidth ] (dcx)
          {{\bf dcx ordering}\\[1ex] \hrule\vspace{1ex} 
          $dcx$-functions of  $(\Phi(B_1),\ldots,\Phi(B_k))$\\
          and shot-noise fields, in
          particular\\          $\cL_{\Phi}(f)$ for   $f\le 0$ or $f\ge 0$};
\node [assoc,text width=0.4\linewidth, below=8ex of dcx] (association) 
          {{\bf negative \& positive association }\\
comparison with respect to the Poisson point
            process\\[1ex]
            \hrule\vspace{1ex}  $\cL_{\Phi}(f)$ for $f\le 0$ or $f\ge 0$};
   \node [method, below left=5ex and -4em of association] (voids)
          {{\bf comparison of void probabilities}\\[0.5ex] \hrule\vspace{1ex} 
          $\cL_{\Phi}(f)$ for  $f\ge 0$}; 
   \node [method, below right=5ex and -4em of  association] (moments)
          {{\bf comparison of moment measures}\\[1ex] \hrule\vspace{1ex} 
          $\cL_{\Phi}(f)$ for  $f\le 0$};
   \node [method, below of=moments] (stat) 
          {{\bf statistical comparison}\\[1ex] \hrule\vspace{1ex} 
          pair correlation function,\\
          Ripley's $K$-function};

  \path [line] (dcx) -| (moments);
  \path [line] (dcx) -| (voids);
  \path [line] (moments) -- (stat);
  \path [line, dashed] (association) -- node {with marginals  $cx$
    ordered to Poisson} (dcx) ;
  \path [line, dashed] (association) |- (voids);
  \path [line, dashed] (association) |- (moments);
\end{tikzpicture}
\caption{\label{blaszcyszyn_f.flowchart1} 
Relationships between clustering comparison methods, 
and some characteristics that allow to compare them. 
Smaller in any type of comparison means that the point process
clusters less.
Recall,
$\cL_\Phi(f):=\EXP{\exp\left[-\int_{\R^d} f(x)\,\allowbreak\Phi(\md x)\right]}$.} 
\end{center}
\end{figure*}
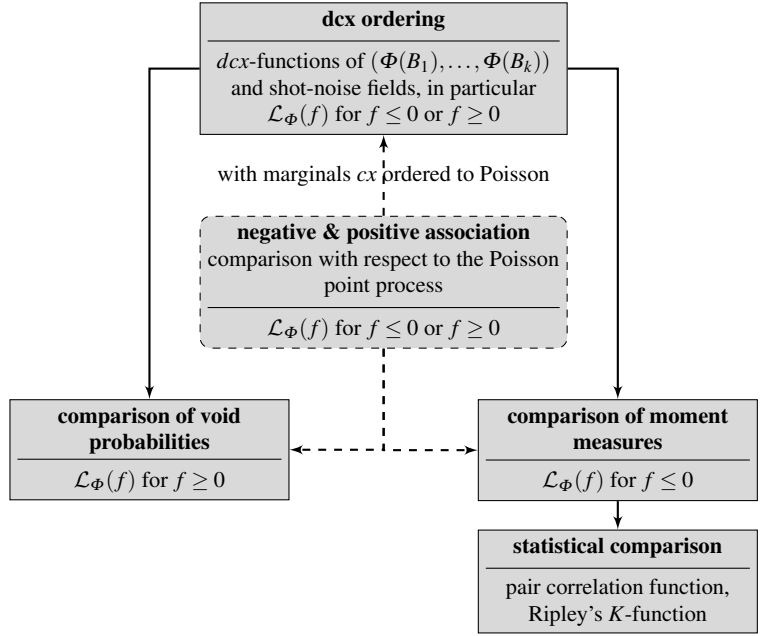

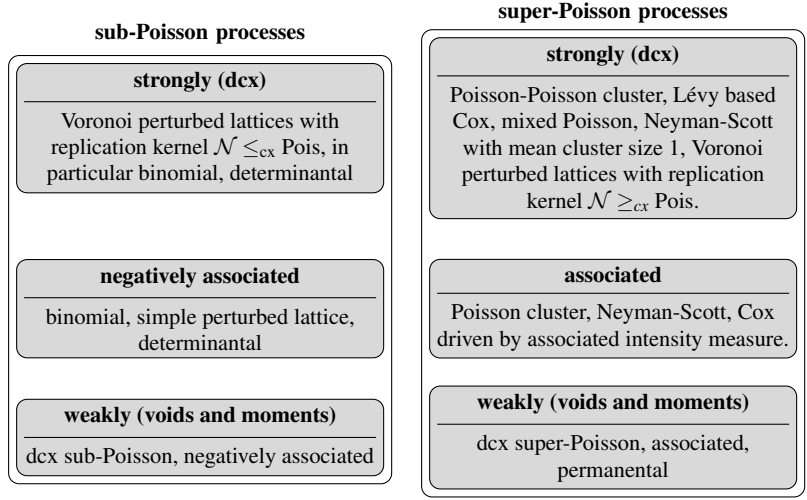
\begin{figure*}[!t]
\begin{center}
\begin{tikzpicture}
\tikzstyle{class} = [draw, rectangle, rounded corners, fill=gray!30,
opacity=1.0, minimum height=2em, text width=0.4\linewidth, text badly centered]
\tikzstyle{bigbox} = [draw,  rectangle,rounded corners]
\node [class] (sub)
{{\bf strongly (dcx)}\\[1ex] \hrule\vspace{1ex}
Voronoi perturbed lattices with replication kernel $\cN\le_{\rm{cx}}\Pois$, 
in particular binomial, determinantal
}; 
\node [class, below=6.5ex of sub] (nassoc)
          {{\bf negatively associated}\\[1ex] \hrule\vspace{1ex}
binomial, simple perturbed lattice, determinantal
}; 
\node[class, below=4ex of nassoc](weaksub){
{\bf weakly (voids and moments)}\\[1ex] \hrule\vspace{1ex}
dcx sub-Poisson, negatively associated
};
\node[bigbox, fit={(sub) (nassoc) (weaksub)}](subPoisson){{\bf
    sub-Poisson processes}\\[43ex]\ };
\node [class, right=2em of sub] (super)
{{\bf strongly (dcx) }\\[1ex] \hrule\vspace{1ex}
        Poisson-Poisson cluster,
        L\'evy based Cox,
        mixed Poisson, 
        Neyman-Scott with mean cluster size~1,
        Voronoi perturbed lattices with replication kernel $\cN\ge_{cx}\Pois$.
}; 
\node [class, right=2em of nassoc] (assoc)
          {{\bf associated }\\[1ex] \hrule\vspace{1ex}
            Poisson cluster, Neyman-Scott,   Cox driven by associated intensity measure.
}; 
\node[class, right=2em of weaksub](weaksuper){
{\bf weakly (voids and moments)}\\[1ex] \hrule\vspace{1ex}
dcx super-Poisson, associated, permanental 
};
\node[bigbox, fit={(super) (assoc) (weaksuper)}](superPoisson)
{{\bf    super-Poisson processes}\\[46ex]\ };
\end{tikzpicture}
\caption{\label{blaszcyszyn_f.flowchart2} 
Some point processes comparable to the Poisson point process using different comparison methods.}  
\end{center}
\end{figure*}

{
\begin{exercise}
Prove the statement of Example~\ref{exe.Cox}.
\end{exercise}
} 

{
\begin{exercise}
Using Hadamard's inequality
prove that determinantal point processes are $\alpha$-weakly 
sub-Poisson.
\end{exercise}
}

\section{Some Applications}
\label{blaszcyszyn_sec:results}

So far we introduced basic notions and results regarding ordering of point processes and
we provided examples of point processes that admit these comparability properties. However,
it remains to demonstrate the applicability of these methods to
random geometric models which will be the goal of this
section. Heuristically speaking, it is possible to easily conjecture the impact
of clustering on various random geometric models, however there is
hardly any rigorous treatment of these issues in the literature. The present 
section shall endeavour to fill this gap by using
the tools of stochastic ordering. We show that one can get 
useful bounds for some quantities of interest on weakly sub-Poisson,
sub-Poisson or negatively associated point processes and that in quite
a few cases these bounds are as good as those for the Poisson point
process. Various quantities of interest are often expressed in terms
of moment measures and void probabilities. This explains the
applicability of our notions of stochastic ordering of point processes
in many contexts. As it might be expected, in this survey of applications,
we shall emphasize breadth more than depth to indicate that many
random geometric models fall within the purview of our
methods. However, despite our best efforts, we would not be able to sketch all possible 
applications. Therefore, we briefly mention a couple of omissions. The
notion of sub-Poissonianity has found usage in at least a couple of
other models than those described below. In \cite{Hirsch2011}, 
connectivity of some approximations of minimal spanning
forests is shown for weakly sub-Poisson point processes. Sans our jargon,
in \cite{Daley05}, the existence of the Lilypond growth model and its
non-percolation under the additional assumption of absolutely
continuous joint intensities is shown for {weakly} sub-Poisson point
processes.

\subsection{Non-trivial Phase Transition in Percolation Models}
\label{blaszcyszyn_ssec:percolation_models}

Consider a stationary  point process $\Phi$ in $\R^d$.  
For a given ``radius'' $r \ge 0$, let us connect 
by an edge any two points of $\Phi$ which are at most at a distance of $2r$
from each other. Existence of an infinite component in the resulting graph is
called {\em percolation} of the graph model based on
$\Phi$. As we have already mentioned in the previous section, 
clustering of $\Phi$ roughly means that the points of $\Phi$
are located in groups being well spaced out. When
trying to find the minimal $r$ for which the graph model based on~$\Phi$
percolates, we observe that points belonging to the same cluster of $\Phi$ will be
connected by edges for some smaller $r$ but points in different
clusters need a relatively high $r$ for having edges between
them. Moreover, percolation cannot be achieved without edges between
some points of different clusters. 
It seems to be 
evident that spreading points
from clusters of $\Phi$ ``more homogeneously'' in the space would
result in a decrease of the radius $r$ for which the percolation takes
place. In other words, \label{blaszcyszyn_heuristic-percol}clustering in a point process $\Phi$ should increase the {\em critical radius} $r_c=r_c(\Phi)$ for the percolation
of the graph model on~$\Phi$, also called  
Gilbert's disk graph or the Boolean model with fixed spherical
grains.

\index{percolation!of Boolean model|seealso{Boolean model}}
\index{percolation!of Boolean model}
\index{Boolean model!percolation}

We have shown in Sect.~\ref{blaszcyszyn_ss.dcx}
that dcx-ordering of point processes  can be used to compare
their clustering properties. Hence, 
the above discussion tempts one to  conjecture that $r_c$ is increasing with respect to dcx-ordering of the underlying point processes; i.e.  
$\Phi_1 \leq_{\text{dcx}} \Phi_2$ implies $r_c(\Phi_1) \leq r_c(\Phi_2)$.
Some numerical evidences gathered in~\cite{dcx-perc}
(where we took Fig.~\ref{blaszcyszyn_f.TwoComponents.subsuperP} from) 
for a dcx-monotone family of
perturbed lattice point processes, were  supportive for this
conjecture.
\begin{figure}[!t]
\begin{center}
\begin{minipage}[b]{1\linewidth}
\vspace{-10ex}
\includegraphics[width=1.\linewidth]{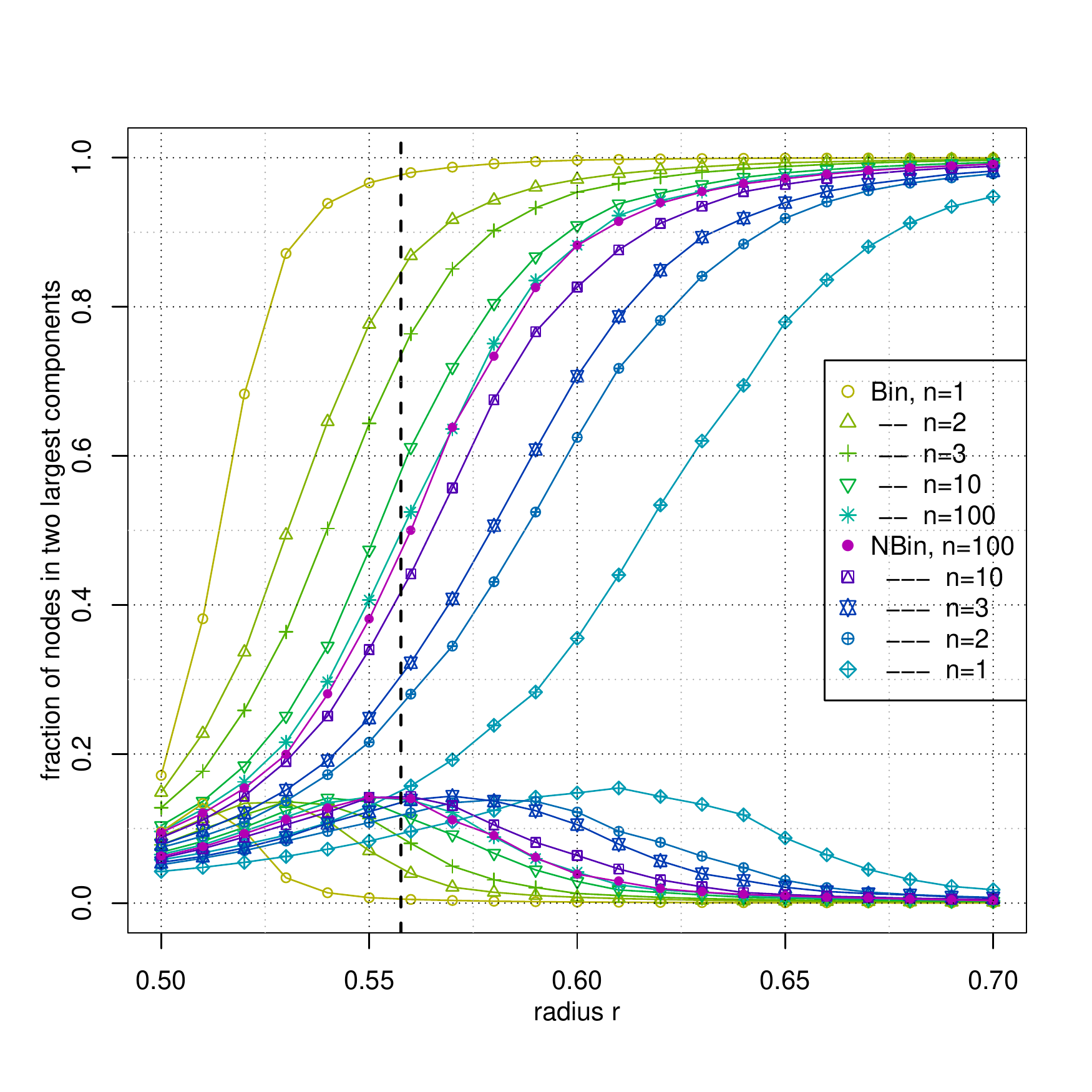}
\end{minipage}
\end{center}
\vspace{-6ex}
\caption{\label{blaszcyszyn_f.TwoComponents.subsuperP}
Mean fractions of nodes in the two largest components
of the sub- and super-Poisson Boolean models
$C(\Phi^{pert}_{\Binom(n,1/n)},r)$ 
and  $C(\Phi^{pert}_{\NBinom(n,1/(1+n))},r)$ ,
respectively, as  functions of $r$.
These families of underlying point processes converge 
to the Poisson point process $\Phi_\lambda$
with intensity $\lambda=2/(\sqrt3)=1.154701$ as $n$ tends to $\infty$;
 cf. Example~\ref{blaszcyszyn_ex.pert_lattice_dcx}.
The dashed vertical line corresponds
to the radius $r=0.5576495$ which is believed to be close 
to the critical radius $r_c(\Phi_\lambda)$.}
\end{figure}

But it turns out that the conjecture is not true in full generality and
a counterexample was  also presented in~\cite{dcx-perc}.
It is a Poisson-Poisson cluster point process, which is known to be
super-Poisson (cf. Example~\ref{blaszcyszyn_ex.super-Poisson-Cox})
whose critical radius is $r_c=0$,  
hence smaller than that of the corresponding Poisson point process, for which $r_c$ is
known to be positive. 
In this Poisson-Poisson cluster  point process, points concentrate  on
some carefully chosen 
larger-scale structure, which itself has good percolation properties. 
In this case, the points concentrate in clusters, however we cannot say that  clusters are well spaced out. Hence,  this example does not 
contradict our initial heuristic explanation of  why an increase of clustering in a
point process should increase the critical radius for the percolation.
It reveals rather that dcx-ordering, while being able to compare
the clustering tendency of point processes, is not capable of 
comparing macroscopic structures of clusters.
Nevertheless, dcx-ordering, and some weaker tools introduced in
Sect.~\ref{blaszcyszyn_sec:methods}, can be used to prove nontrivial phase
transitions for point processes which cluster less than the Poisson process.
In what follows, we will first present some intuitions leading to the
above results and motivating our special focus on moment
measures and void probabilities in the previous sections.

\subsubsection{Intuitions --- Some Non-Standard Critical Radii}
Consider the radii $\underline r_c, \overline r_c$,
which act as lower and upper bounds for the usual
critical radius: i.e. $\underline r_c\le r_c\le \overline r_c$.
We show that clustering acts differently on these bounding radii: 
It turns out that
\index{Boolean model!percolation!nonstandard critical radii}
\index{percolation!nonstandard critical radii}
$$\underline r_c(\Phi_2)\le \underline r_c(\Phi_1)\le r_c(\Phi_1)
\le \overline r_c(\Phi_1)\le \overline r_c(\Phi_2)$$
for $\Phi_1$ having smaller voids and moment measures than $\Phi_2$.
This sandwich inequality tells us that $\Phi_1$ exhibits the usual 
phase transition $0<r_c(\Phi_1)<\infty$, provided $\Phi_2$
satisfies the stronger  conditions $0<\underline r_c(\Phi_2)$ and
$\overline r_c(\Phi_2)<\infty$. Conjecturing that this holds if $\Phi_2$ is a Poisson point process, 
one obtains the  result  on
(uniformly) non-trivial phase transition  for all  weakly sub-Poisson processes~$\Phi_1$,
which has been proved in~\cite{dcx-perc} and will be presented in the subsequent
sections in a slightly different way.

Let $\Phi$ be an arbitrary point process in $\R^d$. 
 Let $W_m = [-m,m]^d$ and define $h_{m,k} : (\R^d)^k \to \{0,1\}$ 
to be the indicator of the event that ${x_1,\ldots,x_k} \in (\Phi \cap
W_m)^k, |x_1| \leq r, \inf_{x \in \partial W_m}|x-x_k| \leq r,
\max_{i\in\{1,\ldots,k-1\}}|x_{i+1} - x_i| \leq r$, where
$\partial W_n$ denotes the boundary of set $W_n$.  
Let $N_{m,k}(\Phi,r) = \sum^{\neq}_{X_1,\allowbreak\ldots,X_k \in \R^d}h_{m,k}(X_1,\ldots,X_k)$ denote the number of distinct self-avoiding paths of length $k$ from 
the origin $o\in\R^2$ to the boundary of the box $W_m$ in the Boolean model and $N_m(\Phi,r) = \sum_{k \geq 1} N_{m,k}(\Phi,r)$ to be the total number of distinct self-avoiding paths 
to the boundary of the box. We define the following ``lower'' critical radius
$$\ur_c(\Phi)  =  \inf \{r : \liminf_m \E N_m(\Phi,r) > 0 \}\,.$$
Note that $r_c(\Phi)=\inf \{r : \lim_{m\to \infty} \P\left(N_m(\Phi,r) \geq 1\right) > 0 \}$, 
with the limit existing because the events $\{N_m(\Phi,r) \geq 1 \}$
form a decreasing sequence in $m$, and by Markov's inequality, we have that 
indeed $\ur_c(\Phi) \leq r_c(\Phi)$ for a stationary point process $\Phi$. 

It is easy to see that $\E N_m(\Phi,r)$  can
be expressed in terms of moment measures, i.e.
 $\E N_m(\Phi,r) =\sum_{k \geq 1} \E N_{m,k}(\Phi,r)$ and 
$$\E N_{m,k}(\Phi,r)=\int_{(\R^d)^k}h_{m,k}(x_1,\ldots,x_k) \alpha_j^{(k)}(\md x_1,\ldots,\md x_k).$$ 
The following result {is obtained  in~\cite{dcx-clust}}.
\begin{proposition}
\label{blaszcyszyn_prop:lower_crit_rad}
Let  $C_j=C(\Phi_j,r)$, $j=1,2$ be two Boolean models with simple point processes of germs
$\Phi_j$, $j=1,2,$ and $\sigma$-finite
$k$-th moment measures $\alpha_j^{k}$ for all $k\ge1$ respectively.
If $\alpha_1^{(k)}(\cdot) \leq \alpha_2^{(k)}(\cdot)$ 
for all $k\ge 1$, then $\ur_c(\Phi_1)\ge \ur_c(\Phi_2).$
In particular, for a stationary, $\alpha$-weakly sub-Poisson point process $\Phi_1$ with unit
intensity, it holds that
$\kappa_d \ur_c(\Phi_1)^d \geq 1$ where $\kappa_d$ is the volume of the unit ball in $\R^d$.
\end{proposition}
%


In order to see void probabilities in action
it is customary to use some discrete  approximations of the continuum
percolation model. For $r > 0, x \in \R^d$, define the following subsets of $\R^d$.
Let
$Q^r = (-\frac{1}{2r},\frac{1}{2r}]^d$ and  $Q^r(x) = x + Q^r.$
We consider the following  discrete graph parametrised by  $n \in \N$. Let
$\mL^{*d}_n=(\Z^d_n, \mE^{*d}_n)$ be the usual close-packed lattice graph scaled down by
the factor $1/n$. It holds that $\Z^d_n  =\frac{1}{n}\Z^d$ 
for the set of vertices  and $\mE^{*d}_n = \{ \langle z_i,z_j \rangle \in
(\Z^d_n)^2 :  Q^{\frac{n}{2}}(z_i) \cap Q^{\frac{n}{2}}(z_j) \neq \emptyset\}$ for the set of edges, where $\Z$ denotes 
the set of integers.

A {\em contour} in $\mL^{*d}_n$ is a minimal collection of vertices such
that any infinite path in $\mL^{*d}_n$
from the origin has to contain one of these
vertices (the minimality condition implies that the removal of any vertex
from the collection will lead to the existence of an infinite path from
the origin without any intersection with the remaining vertices in the
collection). Let $\Gamma_n$ be the set of all contours around the origin in
$\mL^{*d}_n$. For any subset of points $\gamma\subset\R^d$, in particular for paths
$\gamma\in \Gamma_n$, we define  $Q_{\gamma} = \bigcup_{z \in \gamma} Q^n(z)$.

With this notation, we can define the ``upper'' critical radius $\ovr_c(\Phi)$ by
\begin{equation}
\label{blaszcyszyn_e:urc}
\ovr_c(\Phi)= \inf \Bigl\{r>0 : \text{for all}\;  n\ge1, \sum_{\gamma \in
\Gamma_n}\P\left(C(\Phi,r) \cap Q_{\gamma} = \emptyset\right) <  \infty\Bigr\}\,.
\end{equation}
It might be seen as the critical radius corresponding to the phase transition when the discrete model $\mL^{*d}_n=(\Z^d_n, \mE^{*d}_n)$,  approximating
$C(\Phi,r)$ with an arbitrary precision, starts percolating through the Peierls argument. As a consequence, $\ovr_c(\Phi) \geq r_c(\Phi)$ (see \cite[Lemma 4.1]{perc-dcx}).
The following ordering result follows immediately from the definitions.
\begin{corollary}
\label{blaszcyszyn_cor:ord_crit_rad_racs}
Let  $C_j=C(\Phi_j,r)$, $j=1,2$ be two Boolean models with simple point processes of germs
$\Phi_j$, $j=1,2$. 
If $\Phi_1$ has smaller voids probabilities then $\Phi_2$, 
then $\ovr_c(\Phi_1)\le\ovr_c(\Phi_2)$.
\end{corollary}

\begin{remark}
The finiteness of $\ovr_c$ 
is not clear even for Poisson point process  and hence
Corollary~\ref{blaszcyszyn_cor:ord_crit_rad_racs}
cannot be directly used to prove the
finiteness of the critical radii of $v$-weakly sub-Poisson point processes. However, 
the approach based on void probabilities
can be refined, as we shall see in what follows, to conclude the
aforementioned  property. 
\end{remark}

\subsubsection{Percolation of Level-Sets of Shot-Noise Fields}
\label{blaszcyszyn_sec:general_results}
Various percolation problems, including the classical continuum
percolation model considered in the previous section,
can be posed as percolation of some level
sets of shot-noise fields. We say that a \emph{level set percolates}
if it has an unbounded connected subset. We now 
present a useful lemma that can be used in conjunction with other
methods, in particular the 
famous {\em Peierls argument} (cf.~\cite[pp.~17--18]{Grimmett99}), 
to exhibit percolation of the level sets of shot-noise fields
for an appropriate choice of parameters. 
We sketch these arguments in the simple case of
the Boolean model and the application to the SINR model 
in the subsequent sections.
The proofs rely on coupling of a discrete model with the continuum model and
showing percolation or non-percolation in the discrete model using the
above bounds.

In what follows we will be interested in level-sets of shot-noise fields, i.e. sets of the form
$\{y\in S: V_{\Phi}(y) \geq {a}\}$ or $\{y\in S:V_{\Phi}(y) \leq {a}\}$ for some ${a} \in
\R$, where $\{V_{\Phi}(y), y\in S\}$  is a shot-noise  field
generated by some point process $\Phi$  with a non-negative response function $h$ as introduced in Definition~\ref{blaszcyszyn_defn:ISN}. For proving results on percolation of
level-sets, we rely heavily on the bounds from the following
lemma.

\index{shot-noise field!level-sets}
\begin{lemma}[{\upshape{\bf{\cite[Lemma 3.2]{dcx-perc}}}}]
\label{blaszcyszyn_lem:bounds_shot-noise}
Let $\Phi$ be a stationary point process with positive and finite intensity $\lambda$.
Then the following statements are true.
\begin{enumerate}
\item If $\Phi$ is $\alpha$-weakly sub-Poisson,
then 
\begin{equation}
\label{blaszcyszyn_eqn:upper_bd_shot-noise}
 \P\left(\min_{i\in\{1,\ldots,m\}}V_{\Phi}(y_i) \geq {a} \right) \leq e^{-sm{a}} \exp
 \lambda \int_{\R^d}(e^{s\sum_{i=1}^m {h}(x,y_i)} -1) \md x
\end{equation}
for any $y_1,\ldots,y_m \in S $ and $s > 0$. 
\item{If $\Phi$ is $v$-weakly sub-Poisson then, 
\begin{equation}
\label{blaszcyszyn_eqn:lower_bd_shot-noise} 
\P\left(\max_{i\in\{1,\ldots,m\}}V_{\Phi}(y_i) \leq {a}\right) \leq e^{-sm {a}} \exp
 \left(\lambda \int_{\R^d}(e^{s\sum_{i=1}^m {h}(x,y_i)} -1) \md x\right)
\end{equation}
for any $y_1,\ldots,y_m \in S$ and $s > 0$.
}
\end{enumerate}
\end{lemma}

\begin{proof}
Observe that $\sum_{i=1}^mV_{\Phi}(y_i) = \sum_{X \in
\Phi}\sum_{i=1}^m {h}(X,y_i)$ is itself a shot-noise field driven by
the response function $\sum_{i=1}^m {h}(.,y_i).$ Thus if
(\ref{blaszcyszyn_eqn:upper_bd_shot-noise}) and (\ref{blaszcyszyn_eqn:lower_bd_shot-noise})
are true for $m =1$, we can derive the general case as well. 
Using Chernoff's bound, we have that
$\P\left(V_{\Phi}(y) \geq {a}\right) \leq e^{-s {a}}\EXP{e^{sV_{\Phi}(y)}}$
and  $\P \left( V_{\Phi}(y) \leq {a} \right) \leq e^{s {a}}\EXP{e^{-sV_{\Phi}(y)}}$.
The two statements now follow from
Propositions~\ref{blaszcyszyn_p.moments_Laplace} and 
~\ref{blaszcyszyn_p.voids_Laplace} respectively.
\qed
\end{proof}

\subsubsection{$k$-Percolation in the Boolean Model}
\label{blaszcyszyn_sec:k_perc}
\index{Boolean model!k-percolation@$k$-percolation}
By $k$-percolation in a Boolean model, we understand percolation of
the subset of the space covered by at least $k$ grains of the Boolean
model; cf. Definition~\ref{blaszcyszyn_d.Boolean_model}. Our aim is to show that
for sub-Poisson point processes (i.e. point processes that are dcx-smaller 
than the Poisson point process) or negatively associated point processes, the critical connection radius $r$ for $k$-percolation of the Boolean model is non-degenerate,
i.e., the model does not percolate for $r$ too small and percolates
for $r$ sufficiently large. As will be seen in the proof given below, the finiteness of the critical radius $r$ for $k=1$ (i.e. the usual percolation) holds under a weaker assumption of ordering of void probabilities.  

Given a  point processes of germs $\Phi$ on $\R^d$, we define the coverage field
$\{V_{\Phi,r}(x),\allowbreak x\in \R^d\}$ by $V_{\Phi,r}(x) =  \sum_{X \in \Phi} \ind(x \in B_r(X))$, where
$B_r(x)$ denotes the Euclidean ball of radius $r$ centred at $x$. The
$k$-covered set is defined as the following level set.
Let $C_k(\Phi,r) =
\{x\in \R^d : V_{\Phi,r}(x) \geq k \}$. Note that $C_1(\Phi,r)=C(\Phi,r)$ is
the usual Boolean model. For $k\ge 1$, define the {\em critical radius
for $k$-percolation} as $$ r^k_c(\Phi) = \inf \{r :
\P(\mbox{$C_k(\Phi,r)$ percolates}) > 0 \}\,,$$ where, as before,
percolation means existence of an unbounded connected subset. Clearly,
$r_c^1(\Phi) = r_c(\Phi) \leq r_c^k(\Phi)$. As in various percolation
models, the first question is whether $0 < r^k_c(\Phi) < \infty$ ? This 
is known for the Poisson point process (\cite{Gilbert61}) and not for
many other point processes apart from that. The following  result is a
first step in that direction answering the question in affirmative for many point processes.  
\begin{proposition}[{~\upshape{\bf{\cite[Proposition. 3.4]{dcx-perc}}}}]
\label{blaszcyszyn_prop:k_perc_sub-Poisson_point processes}
Let $\Phi$ be a stationary point processes with intensity $\lambda$. For $k \geq 1, \lambda > 0$, there exist constants $c(\lambda)$ and
$C(\lambda,k)$ (not depending on the distribution of $\Phi$) such that
$0 < c(\lambda) \leq r_c^1(\Phi)$ provided  $\Phi$ is $\alpha$-weakly sub-Poisson
and $r_c^k(\Phi) \leq C(\lambda,k) < \infty$ provided
$\Phi$ is $v$-weakly sub-Poisson. Consequently, for $\Phi$ being weakly sub-Poisson,
combining both the above statements, it turns out that 
$$ 0 < c(\lambda) \leq r_c^1(\Phi) \leq r_c^k(\Phi) \leq C(\lambda,k) < \infty. $$
\end{proposition}
\begin{remark}\label{blaszcyszyn_rem:percolation}
{The above result not only shows non-triviality of the critical radius for
stationary weakly sub-Poisson processes  but also provides uniform
bounds. Examples of particular point processes for which this
non-triviality result holds are
determinantal point processes with trace-class integral
  kernels~(cf. Definition~\ref{blaszcyszyn_defn.det} and
Example~\ref{blaszcyszyn_ex.Det-Perm-weakly}) and  
  {Voronoi-}perturbed latices with  convexly
  sub-Poisson replication kernels (cf.
Example~\ref{blaszcyszyn_ex:Voronoi_pert_lattice}).}
For the case of zeros of Gaussian analytic functions which is not covered by us, a non-trivial critical radius for continuum percolation has been shown in \cite{Ghosh12}, where uniqueness of infinite clusters for both zeros of Gaussian analytic functions and the Ginibre point process is also proved.
\end{remark}
\begin{proofsketch}[{\upshape{of Proposition~\ref{blaszcyszyn_prop:k_perc_sub-Poisson_point processes}}}]
A little more notation is required. For $r > 0, x \in \R^d$, define the following subsets of $\R^d$. Let
$Q_r =  (-r,r]^d$ and $Q_r(x) = x + Q_r$. Furthermore, we consider the following discrete graph. 
Let $\Tb^{*d}(r) = (r\Z^d,\Ep^{*d}(r))$ be a close-packed graph on the scaled-up lattice $r\Z^d$; the edge-set is given by $\Ep^{*d}(r) = \{ \langle z_i,z_j \rangle \in (r\Z^d)^2 : Q_r(z_i) \cap Q_r(z_j) \neq \emptyset \}$. Recall that by site-percolation in a graph one means the existence of an infinite connected component in the random subgraph that remains after deletion of sites/vertices as per some random procedure.
{In order to prove the first statement, let $\Phi$ be
  $\alpha$-weakly sub-Poisson and $r > 0$.} Consider the close-packed lattice $\Tb^{*d}({3}r)$. Define the response
function ${h_r}(x,y) = \ind(x \in Q_{{3r/2}}(y))$ and the corresponding
shot-noise field $V^r_{\Phi}(.)$ {on
  $\Tb^{*d}({3}r)$}. 
Note that if $C(\Phi,r)$ percolates then $\{z:V^r_{\Phi}(z)\geq 1\}$ percolates on $\Tb^{*d}({3}r)$
as well. We shall now show that there exists a $r >0$ such that the
latter does not hold true. 
To prove this, we show that the expected
number of paths from $o$ of length $n$ in the random subgraph tends to
$0$ as $n \to \infty$ and Markov's inequality gives that there is
no infinite path (i.e. no percolation) in $\Tb^{*d}({3}r) \cap \{z :
V^r_{\Phi}(z) \geq 1\}.$ There are $(3^d-1)^n$ paths of length $n$
from $o$ in $\Tb^{*d}(2r)$ and the probability that a path $z_i \in
r\Z^d, 1 \leq i \leq n$ is open can be bounded from above as follows. Using~(\ref{blaszcyszyn_eqn:upper_bd_shot-noise})
and some further calculations, we get that
$$\P \left(\min_{i\in\{1,\ldots,n\}}V^r_{\Phi}(z_i) \geq 1 \right) \leq (\exp\{-(s+ (1-e^s)\lambda (3r)^d) \})^n.$$
So, the expected number of paths from $o$ of length $n$ in $\Tb^{*d}(3r) \cap \{z : V^r_{\Phi}(z) \geq 1\}$ is at most $((3^d-1)\exp\{-(s+ (1-e^s)\lambda (3r)^d) \})^n$. 
This term tends to $0$ as $n \to \infty$ for $r$ small enough and $s$ large enough. Since the choice of $r$ depends only on $\lambda$, we have shown that there exists a constant $c(\lambda) > 0$ such that $c(\lambda) \leq r_c^1(\Phi).$
For the upper bound, {let $\Phi$ be $\nu$-weakly
  sub-Poisson and}
consider the close-packed lattice $\Tb^{*d}(\frac{r}{\sqrt{d}})$. Define the response function ${h_r}(x,y)
= \ind(x \in Q_{{\frac{r}{\sqrt{d}}}}(y))$ and the
corresponding additive shot-noise field $V^r_{\Phi}(\cdot).$ 
Note that if the random subgraph $\Tb^{*d}(\frac{r}{\sqrt{d}}) \cap \{z : V^r_{\Phi}(z) \geq {k}\}$ percolates, then $C(\Phi,r)$ also percolates. Then, the following exponential bound is obtained by using (\ref{blaszcyszyn_eqn:lower_bd_shot-noise}) and some more calculations:
$$\P ( \max_{ i\in \{1,\ldots,n \}} V^r_{\Phi}(z_i) \leq {k} - 1 ) \leq {(}\exp\{-((1-e^{-s})\lambda (\frac{r}{\sqrt{d}})^d - s({k} - 1)) \})^n. \, \, \, \, \,$$
It now suffices to use the standard Peierls argument (cf.~\cite[pp.~17--18]{Grimmett99}) to complete the proof. 
\qed 
\end{proofsketch}

{%
For $k=1$; i.e., for the usual percolation in the Boolean model,
we can avoid the usage of the exponential estimates of
Lemma~\ref{blaszcyszyn_lem:bounds_shot-noise}  and work  directly with void
probabilities and factorial moment measures.  
This leads to improved bounds on the critical radius.  
\begin{proposition}[~\upshape{\bf{\cite[Corollary~3.11]{dcx-perc}}}]
\label{cor:phase-transition}

For a stationary weakly sub-Poisson point process $\Phi$  on
$\R^d$, $d\ge 2$,  it holds that

$$0< \frac{1}{(\lam \kappa_d)^{1/d}} \le r_c(\Phi)\le
\sqrt{d}\left(\frac{\log (3^d-2)}{
\lambda}\right)^{1/d}<\infty,$$ 
where $\kappa_d$ is the volume of the unit ball in $\R^d$.
The lower and the upper bounds hold, respectively, for
$\alpha$-weakly $\nu$-weakly  sub-Poisson processes.
\end{proposition}
}
{%
\begin{remark}
Applying the above result to 
an $\alpha$-weakly sub-Poisson point processes with unit intensity we observe that $r_c(\Phi) \geq \kappa_d^{-\frac{1}{d}} \to \infty$ as $d \to \infty.$ This means that in high dimensions,
it holds that $r_c(\Phi) \gg r_c(\Z^d) = \frac{1}{2}$, i.e. like the Poisson point process and even sub-Poisson point processes percolate much worse compared to the Euclidean lattice in high dimensions. 
\end{remark}

However, the question remains open whether the initial heuristic  reasoning saying
that more clustering worsens percolation (cf. p.~\pageref{blaszcyszyn_heuristic-percol}) holds for sub-Poisson point processes. As we have seen and shall see below, sub-Poisson point processes are more tractable than super-Poisson point processes in many respects.
}

\subsubsection{SINR Percolation}
\label{blaszcyszyn_sec:SINR_perc}
\index{SINR model}
\index{SINR model!percolation}
\index{percolation!of SINR model}
For a detailed background about this model of wireless communications, we refer to \cite{subpoisson} and the references therein. Here we directly begin with a formal introduction to the model. We shall work only in $\R^2$ in this section. 

 The parameters of the model are the non-negative numbers $P$ (signal
 power), $N$ (environmental noise), $\gamma$, $T$(SINR threshold) and an attenuation function $\ell :\R^2_+ \to \R_+$  satisfying the following assumptions: $\ell(x,y)  =  l(|x-y|)$ for some continuous function $l : \R_+ \to \R_+$, strictly decreasing on its support, with  $l(0) \geq TN/P$, $l(\cdot) \leq  1$, and  $\int_0^{\infty} x l(x) \md x  <  \infty$. These are exactly the assumptions made in~\cite{Dousse_etal06} and we refer to this article for a discussion on their validity.

Given a point processes $\Phi$, the {\em interference} generated due to the point processes at a location $x$ is defined as the following
shot-noise field $\{ I_{\Phi}(x), x\in \R^d \}$, where
\label{blaszcyszyn_eqn:interference} 
$I_{\Phi}(x) = \sum_{X \in \Phi\setminus\{x\}} l(|X-x|)$.
Define the signal-to-noise ratio (SINR) as follows :
\begin{equation}
\label{blaszcyszyn_eqn:sinr_defn}
\text{SINR}(x,y,\Phi,\gamma) =  \frac{Pl(|x-y|)}{N + \gamma P I_{\Phi\setminus\{x\}}(y)}.
\end{equation}

Let $\Phi_B$ and $\Phi_I$ be two point processes. Furthermore, let $P,N,T > 0$ and $\gamma \geq 0$. The SINR graph is defined as $\Tb(\Phi_B,\Phi_I,\gamma) = (\Phi_B,\Ep(\Phi_B,\Phi_I,\gamma))$ where 
$$\Ep(\Phi_B,\Phi_I,\gamma) = \{ \langle X,Y \rangle~\in~\Phi_B^2 : \min\{SINR(Y,X,\Phi_I,\gamma) , SINR(X,Y,\Phi_I,\gamma)\} > T \}.$$ 
The SNR graph (i.e. the graph without interference, $\gamma = 0$) is defined as $\Tb(\Phi_B) = \Tb(\Phi_B,\emptyset,0)$ and this is nothing but the Boolean model $C(\Phi_B,r_l)$ with $2r_l = l^{-1}(\frac{TN}{P})$. 

 Recall that percolation in the above graphs is the existence of an infinite connected component in the graph-theoretic sense. Denote by $\lambda_c(r) =\lambda (r_c(\Phi_\lambda)/r)^2$ the {\em critical intensity} for percolation of the Boolean model $C(\Pi_{\lambda},r)$. There is much more dependency in this graph than in the Boolean model where the edges depend only on the two corresponding vertices, but still we are able to suitably modify our techniques to obtain interesting results about non-trivial phase-transition in this model. More precisely, we are showing non-trivial percolation in SINR models with weakly sub-Poissonian set of transmitters and interferers and thereby considerably extending the results of \cite{Dousse_etal06}. In particular, the set of transmitters and interferers could be stationary determinantal point processes or sub-Poisson perturbed lattices and the following result still guarantees non-trivial phase transition in the model.
\begin{proposition}[{~\upshape{\bf{\cite[Propositions 3.9 and 3.10]{dcx-perc}}}}]
\label{blaszcyszyn_thm:sinr_poisson_perc}
The following statements are true.
\begin{enumerate}

\item[{\upshape{1.}}] {Let $\lambda > \lambda_c(r_l)$ and let $\Phi$ be a stationary $\alpha$-weakly sub-Poisson point process with intensity $\mu$ for some $\mu > 0$. Then there exists $\gamma > 0$ such that $\Tb(\Pi_{\lambda},\Phi,\gamma)$ percolates.}

\item[{\upshape{2.}}] {Let $\Phi$ be a stationary, $\gamma$-weakly sub-Poisson point processes and let $\Phi_I$ be a stationary $\alpha$-weakly sub-Poisson point process with intensity $\mu$ for some $\mu > 0.$ Furthermore, assume that $l(x) > 0$ for all $x \in \R_+$. Then there exist $P, \gamma > 0$ such that $\Tb(\Phi,\Phi_I,\gamma)$ percolates.}
\end{enumerate}
\end{proposition}

{
\begin{exercise}
Two related percolation models are the $k$-nearest neighbourhood graph ($k$-NNG) and the random connection model. Non-trivial phase transitions for percolation is shown for both models when defined on a Poisson point process in \cite[Sect. 2.4 and 2.5]{Franc_etal07}. We invite the reader to answer the challenging question of whether the methods  of \cite{Franc_etal07} combined with ours could be used to show non-trivial phase transition for weakly sub-Poisson point processes, too. 
\textit{Hint.} The results of \cite{Daley05} on the Lilypond model could be useful to show non-percolation for $1$-NNG. 
\end{exercise}
}

\subsection{U-Statistics of Point Processes}
\label{blaszcyszyn_sec:u_stat}
\index{U statistics}
Denote by $\Phi^{(k)}(\md (x_1,\ldots,x_k))=
\Phi(\md x_1)\left(\Phi\setminus\delta_{x_1}\right)(\md x_2)\dots
\left(\Phi\setminus\sum_{i=1}^{k-1}\delta_{x_i}\right)(\md x_k) $ 
the (empirical) $k$-th order factorial moment measure of $\Phi$. Note that this is a point
process on $\R^{dk}$ with mean measure $\alpha^{(k)}$.
In case when $\Phi$ is simple, $\Phi^{(k)}$ {is simple too and} corresponds to 
the point process of {ordered} $k$-tuples of distinct points of $\Phi$.

In analogy with classical U-statistics, a U-statistic of a point
process $\Phi$ can be defined as the functional $F(\Phi) = \sum_{X
\in \Phi^{(k)}}f(X)$,  for a non-negative symmetric function $f$
(\cite{Reitzner11}). In case when $\Phi$ is infinite one often considers
$F(\Phi\cap W)$, where $W\subset\R$ is a bounded Borel set. 
The reader is referred to~\cite{Reitzner11}  
or~\cite[Sect. 2]{Schulte12} for many interesting U-statistics of point
processes of which two --- subgraphs in a random geometric graph
(\cite[Chap. 3]{Penrose03}) and simplices in a random geometric
complex (see~\cite[Sect. 8.4.4]{Yukich12}) --- are described below.  
\subsubsection{Examples}

\begin{example}[Subgraphs counts in a random geometric graph]
\label{blaszcyszyn_ex.RGG}
\index{random geometric graph}
Let $\Phi$ be a finite point process and $r > 0.$ The random geometric
graph $\Tb(\Phi,r) = (\Phi,\Ep(\Phi,r))$ is defined through its vertices and edges, where the edge set is given by $\Ep(\Phi,r) = \{
(X,Y)\in\Phi : |X-Y| \leq r \}$. For a connected subgraph $\Gamma$ on
$k$ vertices, define $h : \R^{dk} \to \{0,1\}$ by 
$h_{\Gamma,r}(x) = \ind(\Tb(x,r) \cong \Gamma)$, where $\cong$ stands
for graph isomorphism. Now, the number of $\Gamma$-induced subgraphs
in $\Phi$ is defined as  
$$G_r(\Phi,\Gamma) = \frac{1}{k!}\sum_{X \in \Phi^{(k)}}h_{\Gamma,r}(X).$$
Clearly, $G_r(\Phi,\Gamma)$ is a U-statistic of $\Phi.$ {In the special case that  $|\Gamma| = 2$, 
then $G_r(\Phi,\Gamma)$ is
the number of edges.}
\end{example}
\begin{example}[Simplices counts in a random geometric simplex]
\label{blaszcyszyn_ex.simplices}
\index{simplicial complex}
\index{Cech complex@\v Cech complex}
A non-empty family of finite subsets $\Delta(S)$ of a set $S$ is an {\em abstract simplicial complex} 
if for every set $X \in \Delta(S)$ and every subset $Y \subset X$, we have that $Y \in \Delta(S)$. We shall from now on drop the adjective ``abstract''. 
The elements of $\Delta(S)$ are called {\em faces resp. simplices} of the simplicial complex and the dimension of a face $X$ is $|X|-1$. Given a finite point process $\Phi$, one can define the following two simplicial complexes : The {\em \v{C}ech complex} $Ce(\Phi,r)$ is defined as the simplicial complex whose $k$-faces are $\{X_0,\ldots,X_k\} \subset \Phi$ such
that $\cap_i B_r(X_i) \neq \emptyset.$ The {\em Vietoris-Rips complex} \index{Vietoris-Rips complex}
$VR(\Phi,r)$ is defined as the simplicial complex whose $k$-faces are
$\{X_0,\ldots,X_k\} \subset \Phi$ such that $ B_r(X_i) \cap B_r(X_j)
\neq \emptyset$ for all $0 \leq i \neq j \leq k.$ The $1$-skeleton (i.e. the subcomplex consisting of all $0$-faces and $1$-faces) of
the two complexes are the same and it is nothing but the random geometric graph $\Tb(\Phi,r)$ of Example~\ref{blaszcyszyn_ex.RGG}.  Also, the
\v{C}ech complex $Ce(\Phi,r)$ is homotopy 
equivalent to the Boolean model $C(\Phi,r)$. 
The number of $k$-faces in the two
 simplicial complexes can be determined as follows: 
\begin{eqnarray*}
S_k(Ce(\Phi,r)) & = & \frac{1}{k!}\sum_{X \in \Phi^{(k)}}\ind(X
\mbox{ is a $k$-face of $Ce(\Phi,r)$})
\end{eqnarray*}
and similarly for  $S_k(VR(\Phi,r))$.
Clearly, both characteristics
are examples of  U-statistics of $\Phi$.
\end{example}

Note that a U-statistic is an additive shot-noise of the
point process $\Phi^{(k)}$ and this suggests 
the applicability  of our theory to U-statistics. 
Speaking a bit more generally, consider of a family of U-statistics $\Ks$ and
define an additive shot-noise field $\{F(\Phi)\}_{F \in \Ks}$ indexed by $\Ks$.
Why do we consider such an abstraction? Here is an obvious example.

\begin{example}
Consider  $\Ks = \{F_B, F_B(\Phi) = \sum_{X \in \Phi^{(k)}}\ind(f(X) \in B)\}_{B\in B_0(\R_+)}$ for some given non-negative symmetric 
function $f$ defined on
$\R^k$. Then the additive shot-noise field on $\Phi^{(k)}$, indexed by
bounded Borel sets $B\in\R_+$, defined above is nothing but the random field characterising 
the following point process 
on $\R_+$ associated to the U-statistics of~$\Phi$: 
$$\eta_{(f,\Phi)} = \{f(X) : X \in \Phi^{(k)} \}\,$$
in the sense that $F_B(\Phi)=\eta_{(f,\Phi)}(B)$.
(Note that if $|f^{-1}([0,x])| < \infty$ for all $x \in \R_+$, then
$\eta_{\Phi}$ is indeed locally finite and hence a point process and we 
always  assume that $f$ satisfies such a condition.) 
This point process has been studied in~\cite{Schulte12},
in the special case when $\Phi=\Pi_\lambda$ 
is a homogeneous Poisson point process. 
It is shown that if $\lambda\to\infty$, then $\eta_{(f,\Phi)}$
tends to a Poisson point
process with explicitly known intensity measure.
\end{example}

For any U-statistic $F$ and for a bounded window $W\subset\R^d$, we have that
$$\E F(\Phi \cap W)
= \frac{1}{k!} \int_{W}f(x_1,\ldots,x_k)\alpha^{(k)}\md(x_1,\ldots,x_k).$$
Similarly, we can express higher moments of the shot-noise field
$\{F(\Phi)\}_{F \in \Ks}$ by those of $\Phi$. 
With these observations in hand and using Proposition~\ref{blaszcyszyn_thm:isn_rm}, we can state the following result. 
\begin{proposition}
\label{blaszcyszyn_thm:ord_u_stat}
Let $\Phi_1,\Phi_2$ be two point processes with respective factorial moment
measures $\alpha_i^{(k)}$, $i=1,2$ and let $W$ be a bounded
Borel set in~$\R^d$. Consider a family $\Ks$ of U-statistics. 
Then the following statements are true.
\begin{enumerate}
\item[{\upshape 1.}] If $\alpha_1^{(j)}(\cdot) \leq \alpha_2^{(j)}(\cdot)$ 
for all $1 \leq j \leq k$, then 
$$\E(F_1(\Phi_1 \cap W)F_2(\Phi_1 \cap
W)\ldots F_k(\Phi_1 \cap W)) \leq \E(F_1(\Phi_2 \cap W)F_2(\Phi_2 \cap
W)\ldots F_k(\Phi_2 \cap W))$$ 
for any $k$-tuple of U-statistics
$F_1,\ldots,F_k\in\Ks$. In particular, for any given $F\in\Ks$ based
on a non-negative symmetric function~$f$ it holds that 
$\alpha_{\eta_{(f,\Phi_1)}}^{(k)}(\cdot) \leq \alpha_{\eta_{(f,\Phi_2)}}^{(k)}(\cdot)$,
where $\alpha_{\eta_{(f,\Phi_i)}}^{(k)}$ is the $k$-th order factorial
moment measure of~$\eta_{(f,\Phi_i)}$.

\item[{\upshape 2.}] If $\Phi_1 \leq_{\rm{dcx}} \Phi_2$, then $\{F(\Phi_1 \cap W)\}_{F \in \Ks} \leq_{\rm{idcx}} \{F(\Phi_2 \cap W)\}_{F \in \Ks}$ and in particular $\eta_{(f,\Phi_1)} \leq_{\rm{idcx}} \eta_{(f,\Phi_2)}$. 
\end{enumerate}
\end{proposition}

\subsubsection{Some Properties of Random Geometric Graphs}
\label{blaszcyszyn_sec:Rgg}
The subgraph count $G_r(\cdot,\cdot)$ considered in
Example~\ref{blaszcyszyn_ex.RGG} is only a particular example of a U-statistic but
its detailed study in the case of the Poisson point process
(see \cite[Chap. 3]{Penrose03}) was the motivation to derive results
about subgraph counts of a random geometric graph over other point
processes (\cite{Adler12}). Here, we explain some simple
results about clique numbers, maximal degree and chromatic number that
can be deduced as easy corollaries of ordering of subgraph counts
known due to Proposition~\ref{blaszcyszyn_thm:ord_u_stat}.  

Slightly differing from \cite{Penrose03}, we
consider the following asymptotic regime for a 
stationary point process $\Phi$ with unit intensity. We look at the properties of
$\Tb(\Phi_n,r_n)$, $n \geq 1$, where $\Phi_n = \Phi \cap W_n$ with $W_n =
[-\frac{n^{\frac{1}{d}}}{2},\frac{n^{\frac{1}{d}}}{2}]^d$ and a radius
regime $r_n.$ To compare our results with those of \cite{Penrose03},
replace the $r_n^d$ factor in our results by $nr_n^d.$ Detailed
asymptotics of $G_n(\Phi_n,\Gamma) = G_{r_n}(\Phi_n,\Gamma)$ for
general stationary point processes have been studied in \cite[Sect.
3]{Adler12}.  

Let $\mC_n = \mC_n(\Phi),\De_n = \De_n(\Phi),\X_n = \X_n(\Phi)$
denote the size of the largest clique, maximal vertex
degree and chromatic number of $\Tb(\Phi_n,r_n)$, respectively. Heuristic
arguments for these quantities say that they should increase with more
clustering in the point process. We give a more formal statement
of this heuristic at the end of this section.  

Let $\Gamma_k$ denote the complete graph on $k$ vertices and
$\Gamma^{'}_1,\ldots,\Gamma^{'}_m$ be the maximum collection of
non-isomorphic graphs on $k$ vertices having maximum degree $k-1.$
Then for $k \geq 1$, we have the following two equalities and the
graph-theoretic inequality that drive the result following them:
$\{\mC_n < k\} = \{G_n(\Phi_n,\Gamma_k) = 0\}$, $\{\De_n < k-1\}
= \cap_{i=1}^m\{G_n(\Phi_n,\Gamma^{'}_i) = 0\}$,  $\mC_n \leq \X_n \leq \De_n + 1$.
\begin{corollary}\label{blasz.coro}
Let $\Phi$ be a stationary $\alpha$-weakly sub-Poisson point process with unit intensity. 
If $nr_n^{d(k-1)} \to 0$, then
$$ \lim_{n \to \infty} \P(\mC_n < k) = \lim_{n \to \infty}  \P(\De_n < k - 1) = \lim_{n \to \infty}  \P(\X_n < k) = 1.$$
\end{corollary}
\begin{proofsketch}
To prove the result for $\mC_n$, due to Markov's inequality and the fact that
 $\E G_n(\Phi_n,\Gamma) \leq \E G_n(\Pi_{1_n},\Gamma)$ (see
 Proposition~\ref{blaszcyszyn_thm:ord_u_stat}), it suffices to show that
 $\E G_n(\Pi_{1_n},\Gamma) \to 0$ for $nr_n^{d(k-1)} \to 0$ and any
 graph $\Gamma$ on $k$ vertices. This is already known
 from \cite[Theorem 6.1]{Penrose03}. The proof for $\De_n$ is similar
 and it also proves the result for $\X_n$.
\qed  
\end{proofsketch} 

Note that if $k =2$, the results of Corollary~\ref{blasz.coro} implies that
$\lim_{n \to \infty} \P(\mC_n = 1) = 1$ by using the trivial lower
bound of $\mC_n \geq 1$, and analogously for the other two quantities. To
derive a similar result for $k \geq 2,$ we need variance bounds for
$G_n(\cdot,\cdot)$ to use the standard second moment method.
These variance bounds for $G_n(\cdot,\cdot)$ are available in the case of
negatively associated point processes (see \cite[Sect.~3.4]{Adler12}).    

Further, for point processes with $\alpha^{(k)}$ admitting a
continuous density
$\alpha^{(k)}(\md(x_1,\ldots,\allowbreak x_k))=\rho^{(k)}(x_1,\ldots,x_k)\md
x_1,\ldots,\md x_k$ in the neighbourhood of  $(0,\ldots,0)$, such that 
$\rho^{(k)}(0,\allowbreak\ldots,0) = 0$ (for example, $\alpha$-weakly
sub-Poisson point processes such as the Ginibre point process,
perturbed lattice, zeros of Gaussian analytic function et al.) we know
from \cite[Sect.~3.2]{Adler12} that
$\frac{\EXP{G_n(\Phi_n,\Gamma_k)}}{nr_n^{d(k-1)}} \to 0.$ Using this
result, we can show that $\lim_{n \to \infty} \P(\mC_n < k) \to 1$
even for $nr_n^{d(k-1)} \to \lambda > 0.$ This is not true for the
Poisson point process (see \cite[Theorem 6.1]{Penrose03}). Thus, we
have the following inequality for point processes with
$\rho^{(k)}(0,\allowbreak\ldots,0) = 0$:
\begin{equation}\label{blasz.equ.p}
\lim_{n \to \infty} \P(\mC_n(\Pi_1) < k) \leq \lim_{n \to \infty} \P(\mC_n(\Phi) < k).
\end{equation}
This inequality can be easily concluded from the fact that for
radius regimes with $nr_n^{d(k-1)} \to \lambda \geq 0,$ 
the expression on the right-hand side of~(\ref{blasz.equ.p}) is equal to $1$
while for the radius regime $nr_n^{d(k-1)} \to \infty,$ the expression on the left-hand side
of~\ref{blasz.equ.p} is equal to $0$ (see \cite[Theorem 6.1]{Penrose03}). In vague terms, we can
rephrase the above inequality as that $\mC_n(\Pi_1)$ is ``stronger
ordered'' than $\mC_n(\Phi)$ in the limit, i.e., the Poisson point
process is likely to have a larger clique number than a point process
with $\rho^{(k)}(0,\ldots,0) = 0.$  Similar ``strong ordering''
results for $\De_n$'s and $\X_n$'s matching well with heuristics can
also be derived.   

\subsection{Random Geometric Complexes}
\label{blaszcyszyn_sec:rand_geom_complexes}
We have already noted in Example~\ref{blaszcyszyn_ex.simplices}
that the number of $k$-faces in \v{C}ech and Vietoris-Rips complexes
on point processes are U-statistics. In this section we will
further describe the topological properties of these random geometric
complexes. In the same manner as simplicial complexes are considered
to be topological extensions of graphs, so are random geometric
complexes to random geometric graphs. Random geometric graphs on
Poisson or binomial point processes are a well-researched subject with
many applications (see \cite{Penrose03}). Motivated by
research in topological data analysis (\cite{Carlsson09,Ghrist08}) and
relying on results from random geometric graphs, random geometric
complexes on Poisson or binomial point processes have been studied
recently \cite{Kahle11}. In \cite{Adler12}, the investigation of the topology
of random geometric complexes has been extended to a wider class of stationary point processes using tools from stochastic ordering of point processes. 
However, we shall content ourselves with just explaining one of the key phase-transition results 
given in \cite{Adler12}.  

One of the first steps towards the understanding of \v{C}ech and
Vietoris-Rips
complexes is to understand the behaviour of their {\em Betti numbers}
$\beta_k(\cdot), k \geq 0$ as  functions of $r$. 
Informally speaking, the $k$-th ($k \geq 1$)
Betti number counts the number of $k+1$-dimensional holes in the
appropriate Euclidean embedding of the simplicial
complex. The $0$-th Betti number is the number of connected components
in the simplicial complex, $\beta_1$ is the number of two-dimensional 
or "circular" holes, $\beta_2$ is the number of three-dimensional
voids, etc.  
If $\beta_0(\cdot) = 1$, then we say that the
simplicial complex 
is connected.
Unlike simplicial counts, Betti numbers are not
U-statistics. 

Regarding the dependence of the Betti number  $\beta_k(r)$ 
of \v{C}ech complexes on $r$, for $k\ge1$, unlike in earlier
percolation models, there are {\em two phase-transitions} happening 
in this case: $\beta_k(r)$  goes from zero
to positive (the complex ``starts creating'' $k+1$-dimensional holes) 
and, alternatively, it goes from positive to
zero (the holes are ``filled in'').  
If one were to think about the relative behaviours of
the Betti numbers $\beta_k(\cdot)$, $k \geq 1$, of the \v{C}ech complexes on
two point process $\Phi_1, \Phi_2$ where $\Phi_1$ is "less clustered"
than $\Phi_2$, then it should be possible 
that the first threshold decreases with clustering and the 
second threshold increases with clustering. 
Indeed, depending on the strength
of the result we require, weak sub-Poissonianity or negative
association turn out to be the right notion to prove the above
heuristic more rigorously.  



Let $\Phi$ be a stationary weakly sub-Poisson point process with unit
intensity and let $W_n = n^{1/d}[-\frac{1}{2},\frac{1}{2}]^d$. Let $r_n
\geq 0, n \geq 1$ be the corresponding radius regime with $\lim_n r_n
\in [0,\infty]$. Based on whether $\lim_n r_n$ is $0,\infty$ or a
constant between $0$ and $\infty$, we shall get different scaling limits for the Betti
numbers. Under certain technical assumptions, there is a function
$f^k(\cdot)$ depending on the joint intensities of the point process
($f^k(r) \to 0$ as $r \to 0$ for point processes for which 
$\rho^{(k)}(0,\ldots,0) = 0$ and otherwise $f^k(\cdot) \equiv 1$) such that the following statements hold for $k \geq 1$. 
1. If  
$$r_n^{d(k+1)}f^{k+2}(r_n) = o(n^{-1}) \, \, \, \mbox{or} \, \, \, r_n^d = \omega(\log n),$$
then with high probability $\beta_k(Ce(\Phi \cap W_n,r_n)) = 0.$ 2. Let
$\Phi$ be negatively associated. If 
$$r_n^{d(k+1)}f^{k+2}(r_n) = \omega(n^{-1}) \, \, \, \mbox{and} \, \, \, r_n^d = O(1), $$
then with high probability  $\beta_k(Ce(\Phi_n,r_n)) \neq 0$.

To see how the above statements vindicate the heuristic described before, consider the first threshold for the appearance of Betti numbers (with high probability). In the Poisson case ($f^k(\cdot) \equiv 1$) this threshold is exactly $r_n \to 0$ and $nr_n^{d(k+1)} \to \infty$, whereas for weakly sub-Poisson point processes with $\rho^{(k)}(0,\ldots,0) = 0$ ($f^k(r) \to 0$ as $r \to 0$), this is at least $r_n \to 0$ and $nr_n^{d(k+1)}  \to \infty$, i.e.,  the threshold is larger for these weakly sub-Poisson point processes. For specific point processes such as the Ginibre point process or the zeros of Gaussian analytic functions ($f^k(r) = r^{k(k-1)}$), this threshold is only at $r_n \to 0$ and $nr_n^{(k+1)(k+2)} \to \infty$ which is much larger than that of the Poisson point process.  

Now if we consider the second threshold, when Betti numbers vanish, then for the Ginibre point process or the zeros of Gaussian analytic functions, with high probability $\beta_k(Ce(\Phi \cap W_n,r_n)) = 0$ for $r_n^d = \omega(\sqrt{\log n})$. This is of strictly smaller order compared to the Poisson case where $r_n^d = \omega(\log n)$. For other weakly sub-Poisson point processes, the above results imply that the order of the radius threshold for vanishing of Betti numbers cannot exceed $r_n^d = \omega(\log n)$, i.e. that of the Poisson point process. Hence, this second threshold is smaller for weakly sub-Poisson point processes. Negative association just assures positivity of Betti numbers for the intermediate regime of $r_n \to r \in (0,\infty)$.
Now if we consider the second threshold, when Betti numbers vanish, then for the Ginibre point process or the zeros of Gaussian analytic functions, with high probability $\beta_k(Ce(\Phi \cap W_n,r_n)) = 0$ for $r_n^d = \omega(\sqrt{\log n})$. This is of strictly smaller order compared to the Poisson case where $r_n^d = \omega(\log n)$. For other weakly sub-Poisson point processes, the above results imply that the order of the radius threshold for vanishing of Betti numbers cannot exceed $r_n^d = \omega(\log n)$, i.e. that of the Poisson point process. Hence, this second threshold is smaller for weakly sub-Poisson point processes. Negative association just assures positivity of Betti numbers for the intermediate regime of $r_n \to r \in (0,\infty)$.

Barring the upper bound for the vanishing of Betti numbers, similar results hold true for the Vietoris-Rips complex, too. One can obtain asymptotics for Euler characteristic of the \v{C}ech complex using the Morse-theoretic point process approach. This and asymptotics for Morse critical points on weakly sub-Poisson or negatively associated point processes are also obtained
in \cite{Adler12}.  

Furthermore, analogous to percolation in random geometric graphs, one can study percolation in random \v{C}ech complexes by defining the following graph. Any two $k$-faces of the \v{C}ech complex are said to be connected via an edge if both of them are contained in a $(k+1)$-face of the complex. In the case of $k = 0$, the graph obtained is the random geometric graph. The existence of non-trivial percolation radius for any $k \geq 1$ is guaranteed by Proposition \ref{blaszcyszyn_prop:k_perc_sub-Poisson_point processes} provided the point process $\Phi$ is weakly sub-Poisson. This question was raised in \cite[Sect. 4]{Kahle11} for the Poisson point process.

\subsection{Coverage in the Boolean Model}
\label{blaszcyszyn_sec:coverage}
\index{Boolean model!coverage}
In previous sections of this Chapter, we looked at percolation and connectivity aspects of the Boolean model. Yet another aspect of the Boolean model
that has been well studied for the Poisson point process is coverage~\cite{Hall88}. Here, one is concerned with the volume of
the set $C_k(\Phi,r)$ defined in Sect.~\ref{blaszcyszyn_sec:k_perc}. The heuristic is 
again that the volume of the $1$-covered region should decrease
with clustering, but the situation would reverse for the $k$-covered region for large enough $k$.  We are in a position to state a more formal
statement once we introduce still another definition. We say that two {discrete}
random variables $X,Y$ are ordered in {\em uniformly convex
variable order} (UCVO)($X \leq_{\text{uv}} Y$) if their respective density functions $f,g$ satisfy the following conditions: $\text{supp}(f) \subset \text{supp}(g)$, $f(\cdot)/g(\cdot)$ is an unimodal function but their respective distribution functions are not
ordered, i.e. $F(\cdot) \nleqslant G(\cdot)$ or vice-versa (see \cite{Whitt1985}) and where $\text{supp}(\cdot)$  
denotes the support of a function.

%

\begin{proposition}[{~\upshape{\bf{\cite[Proposition 6.2]{dcx-clust}}}}]
\label{blaszcyszyn_prop:k_coverage}
Let $\Phi_1$ and $\Phi_2$ be two simple, stationary point processes such that
$\Phi_1(B_O(r)) \leq_{uv} \Phi_2(B_O(r)$ for $r \geq 0$.
Then there exists $k_0 \geq 1$ such that for any 
bounded Borel set $W \subset \R^d$ it holds that 
\begin{eqnarray*}
 & \E \nu_d(C_k(\Phi_1,r) \cap W ) \geq \E \nu_d(C_k(\Phi_2,r) \cap W) & \quad \text{ for all } k: 1 \leq k \leq k_0, \, \, \\
 \mbox{and}\\
 &  \E\nu_d(C_k(\Phi_1,r) \cap W ) \leq \E\nu_d(\|C_k(\Phi_2,r) \cap W) & \quad \text{ for all } k > k_0,.
\end{eqnarray*}
\end{proposition}

For $\Phi_1$, we can take any of the sub-Poisson perturbed lattices presented in Example \ref{blaszcyszyn_ex:Voronoi_pert_lattice} or a determinantal point process (see Definition \ref{blaszcyszyn_defn.det}) and $\Phi_2$ as a Poisson point process. We can also take $\Phi_1$ to be a Poisson point process and $\Phi_2$ to be any of the super-Poisson perturbed lattices presented in Example \ref{blaszcyszyn_ex:Voronoi_pert_lattice} or a permanental point process (see Definition \ref{blaszcyszyn_defn.perm}). 

Note that for a stationary point process $\Phi$, we have that
$$\E\nu_d(C_k(\Phi,r) \cap W ) = \int_{W} \P(\Phi(B_x(r)) \geq k) \md x = \nu_d(W) \P(\Phi(B_O(r)) \geq k).$$
It is easy to derive from the above equation that the expected $1$-covered region of ${\nu}$-weakly sub-Poisson point processes is larger than that of the Poisson point process \cite[Sect.~6.1]{snorder}.     

The question of coverage also arises in the SINR model of Sect.~\ref{blaszcyszyn_sec:SINR_perc}. We shall not delve further into this question other than remarking that in a certain variant of the SINR model, it has been shown that the coverage and capacity which is defined as $\log(1 + SINR)$ increase with increase in dcx-ordering~\cite[Sect.~5.2.3]{Yogesh_thesis}.

\section{Outlook}
\label{blaszcyszyn_sec:outlook}

Let us mention some possible directions for future work. While several examples of point processes comparable to the Poisson point process were presented, a notable absentee from our list
are  Gibbs point processes, which should appear in the context of
modelling of clustering phenomena. In particular, some Gibbs
hard-core point processes are expected to be sub-Poisson. 
Bounds for the probability generating functionals with 
estimates for Ripley's K-function and the intensity and higher
order correlation functions for  some  stationary
locally stable Gibbs point process are given
in~\cite{ss12b}.  
Also, some geometric structures on specific Gibbs point processes have
already been considered (see e.g. \cite{Coupier12,Schreiber12}), and
these processes perhaps could serve as new reference processes,
replacing in this role the Poisson point process. As for today we are
not aware of  any such results. 

The question of other useful orders for comparison of point processes also is worthy of investigation. In \cite[Sect.~4.4]{Yogesh_thesis}, it has already been shown that related orders of supermodularity and componentwise convexity are not suitable orders whereas convexity could be useful. It might be interesting to study convex ordering of point processes.  

Though we have presented applications to various geometric models, there are many other questions such as the ordering of critical radii for percolation, uniqueness of giant component in continuum percolation (see \ref{blaszcyszyn_rem:percolation}), concentration inequalities for more general functionals (see Proposition \ref{blaszcyszyn_prop:conc_ineq}), asymptotic analysis of other U-statistics along the lines of Sect. \ref{blaszcyszyn_sec:rand_geom_complexes} and \ref{blaszcyszyn_sec:Rgg}, etc., which to be investigated. 

\begin{acknowledgement}
This paper is based on research supported in part
by Israel Science Foundation 853/10, AFOSR
FA8655-11-1-3039 and FP7-ICT-318493-STREP.
\end{acknowledgement}

\bibliographystyle{spmpsci}
\bibliography{references_2013-12-12_new}
%
%

\end{document}